\documentclass{article}
\usepackage[paperheight=235mm,paperwidth=155mm,textwidth=125mm,textheight=199mm]{geometry}
\usepackage{amsmath,amsfonts,amsthm,amssymb}
\usepackage[all,knot]{xy}
\usepackage{slashed}
\usepackage{graphicx}
\usepackage{psfrag}
\usepackage{tikz}
\usepackage{color}
\usepackage{dsfont}
\usepackage{verbatim}
\usepackage{float}
\usepackage{enumitem}
\usetikzlibrary{arrows}
\usepackage{fancyhdr}
\usepackage{graphicx}
\usepackage{pinlabel}
\usepackage[abs]{overpic}
\usepackage{xcolor,varwidth}

\usepackage[all]{xy}

\newcommand{\de}{ \textrm{d} }

\newtheorem{theorem}{Theorem}%[chapter]
\newtheorem{proposition}[theorem]{Proposition}
\newtheorem{hypothesis}[theorem]{Hypothesis}
\theoremstyle{definition}
\newtheorem{example}[theorem]{Example}
\newtheorem{algorithm}[theorem]{Algorithm}
\newcommand\numberthis{\addtocounter{equation}{1}\tag{\theequation}}
\newtheorem{corollary}[theorem]{Corollary}

\newtheorem{remark}[theorem]{Remark}
\newtheorem{definition}[theorem]{Definition}

\newtheorem*{proposition*}{Proposition}

\def\ra{\rightarrow}

\def\txtd{{\textnormal{d}}}

\def\R{\mathbb{R}}

\def\N{\mathbb{N}}

\def\1{\mathds{1}}

\def\cK{\mathcal{K}}

\def\cM{\mathcal{M}}

\def\cO{\mathcal{O}}

\def\dist{\operatorname{dist}}

\newcommand{\fa} {\quad \text{for all }\,}

\begin{document}

\title{Early-warning signals for bifurcations in \\ random dynamical systems with bounded noise}
\author{Christian Kuehn\footnote{Technical University of Munich, Faculty of Mathematics, Boltzmannstr. 3, 85748 Garching bei M{\"u}nchen, Germany}, Giuseppe Malavolta\footnote{Department of Mathematics, Imperial College London, 180 Queen's Gate, London SW7 2AZ, United Kingdom}, and Martin Rasmussen\footnote{Department of Mathematics, Imperial College London, 180 Queen's Gate, London SW7 2AZ, United Kingdom}}

\maketitle

\begin{abstract}
  We consider discrete-time one-dimensional random dynamical systems
  with bounded noise, which generate an associated set-valued dynamical system. We provide
  necessary and sufficient conditions for a discontinuous bifurcation of a
  minimal invariant set of the set-valued dynamical system in terms
  of the derivatives of the so-called extremal maps. We propose an algorithm for reconstructing the derivatives of the
  extremal maps from a time series that is generated by iterations of the original
  random dynamical system. We demonstrate that the derivative reconstructed for different parameters
  can be used as an early-warning signal to detect an upcoming bifurcation, and
  apply the algorithm to the bifurcation analysis of the stochastic return map
  of the Koper model, which is a three-dimensional multiple time scale
  ordinary differential equation used as prototypical model for the formation of
  mixed-mode oscillation patterns. We apply our algorithm to data
  generated by this map to detect an upcoming transition.
\end{abstract}



%%%%%%%%%%%%%%%%%%%%%%%%%%%%%%%%%%%%%%%%%%%%%%%%%%%%%%%%%%%%%%%%%%%%%%%%%%%%%%%%%%%%%%%%%%%%
\section{Introduction}

There has been a steadily increasing interest into dealing with sudden changes in the behaviour of dynamical systems, sometimes referred to as \emph{critical transitions} and \emph{tipping points} in the applied sciences. Many critical transitions can be modelled mathematically using differential equations involving \emph{bifurcations}. For instance, medical conditions such as asthma attacks and epileptic seizures can change quickly from regular to irregular behaviour, the financial markets are known to suddenly break trends in a crisis, and climate conditions and ecological environments can change rather abruptly.

The understanding of the dynamical behaviour near critical transitions is very important, since it enables human interaction to attenuate or control the consequences of critical transitions, but predicting critical transitions before they are reached is notably hard. Recent results by scientists working in different fields of applications suggest the existence of generic early-warning signals (such as autocorrelation and variance) that could indicate for a large class of systems if a critical transition is being approached \cite{Schefferetal,KuehnCT2}.

In low-dimensional autonomous (time-independent) dynamical systems, critical transitions have been extensively studied in the context of bifurcation theory, by explaining and classifying ways in which attractors can lose stability and give rise to new types of behaviours, from simple transitions between stationarity and oscillations, to transitions between predictable and intrinsically unpredictable (chaotic) dynamics \cite{Chow_96_1,Kuznetsow_95_1}. There is a central assumption of very slow (adiabatic) variation of parameters in the classical bifurcation theory, and the problem of critical transitions in complex systems and their early-warning signals motivates the need to deal with transitions in nonautonomous (time-dependent) and random dynamical systems \cite{Kloeden_11_2,ArnoldSDE}. Critical transitions in form of rate-induced tipping in nonautonomous dynamical systems have recently been studied in \cite{Ashwin_17_1,Ritchie_16_1}

In this paper, we consider discrete-time random dynamical systems with bounded noise, which give rise to a set-valued dynamical system. We study bifurcations of the set-valued dynamical system in form of discontinuous changes in the minimal invariant sets (which correspond to critical transitions), and we address the question of an early warning to predict such transitions. We provide an algorithm that approximates the derivatives of the extremal mappings of one-dimensional set-valued dynamical systems from a time series generated by the corresponding random dynamical system with bounded noise. The use of derivatives of extremal maps as early-warning signals has not been established before. We show that the derivative is equal to $1$ at a bifurcation point necessarily, and provides a universal threshold for this reason, which does not depend on the type of system or the type of bounded noise.

We apply our theoretical results to the return map in the Koper model \cite{Koper}. This model is a prototypical example exhibiting mixed-mode oscillations (MMOs) \cite{Desrochesetal}. It is a system of ordinary differential equations  with one fast and two slow variables. Formally, a generic return map to a codimension-one section would be a two-dimensional map but the strong contraction in one direction implies that the system has an effectively one-dimensional return map \cite{Guckenheimer8}. We study a bounded noise perturbation of the Koper model leading to a random return map. In particular, we are interested in the detection of bifurcations that induce a change in the MMO pattern. We demonstrate using numerical simulation and analysis of the resulting time series, that the derivatives of the extremal maps can be used as an early-warning sign for the changing patterns.

This paper is organized as follows. In Section~\ref{sec:bifurcation}, we provide necessary and sufficient conditions for bifurcations of minimal invariant sets. We formulate the conditions in terms of the derivative of extremal maps, which define the shape of the graph of the set-valued map. In Section~\ref{sec:algorithm}, we present an algorithm for estimating the derivatives of the extremal maps from a time series generated by the corresponding random dynamical system. We prove also a probability statement about the reliability of the estimate containing an explicit dependence on the sample size. In Section~\ref{sec:koper}, we apply our method to predict changes in MMOs in the Koper model.

\emph{Acknowledgements.} Christian Kuehn was partially supported by a Lichtenberg Professorship of Volkswagen\-Stiftung. Giuseppe Malavolta and Martin Rasmussen were supported by an EPSRC Career Acceleration Fellowship EP/I004165/1. Martin Rasmussen was supported by funding from the European Union's Horizon 2020 research and innovation programme for the ITN CRITICS under Grant Agreement Number 643073.

\section{Bifurcations of set-valued mappings}
\label{sec:bifurcation}

Consider the dynamics of continuous mappings $f:\R\to \R$ perturbed by additive bounded noise, given by the random difference equation
\begin{displaymath}
  x_{n+1} = f(x_n) + \xi_n\,,
\end{displaymath}
where $\{\xi_n\}_{n\in\N}$ is a sequence of i.i.d.~random variables with values in a compact set $K\subset \R$. The collective behavior of all
future trajectories can then be studied via the set-valued mapping $F:\R\to\cK(\R)$, defined by $F(x):= \{f(x)+y\in\R: y\in K\}$,
where $\cK(\R)$ is the set of all nonempty compact subsets of $\R$, equipped with the Hausdorff distance $h:\cK(\R)\times \cK(\R)\to \R_0^+$.

Note that for two nonempty sets $A,B\subset \R$ and $x\in \R$, the \emph{distance} of $x$ to $A$ is given by $\dist(x,A) := \inf \{|x-y| : y\in A\}$, and the \emph{Hausdorff semi-distance} of $A$ and $B$ is defined by $d(A|B) := \sup \{\dist(x,B): x\in A\}$. Then the \emph{Hausdorff distance} of $A$ and $B$ is given by $h(A,B):= \max\{d(A|B),d(B|A)\}$.

In this paper, we study set-valued mappings $F:\R\to\cK(\R)$ without making reference to a deterministic mapping $f$ as above. We assume throughout that set-valued mappings $F:\R\to\cK(\R)$ satisfy the following conditions.

\begin{hypothesis}[on the set-valued mapping $F$]
  We assume the following hypotheses:
  \begin{itemize}
    \item[(H1)] $F(x)$ is an interval for all $x \in \R$.
    \item[(H2)] There exists an $\varepsilon >0$ such that $F(x)$ contains an interval of length $\varepsilon$ for all $x\in\R$.
    \item[(H3)] The so-called \emph{extremal maps} $f^-:\R\to\R$ and $f^+:\R\to\R$, given by
      \begin{equation}\label{botmap}
         f^-(x)  :=  \min F(x)  \quad \mbox{and} \quad f^+(x) =  \max F(x)\,,
      \end{equation}
      are continuously differentiable.
  \end{itemize}
\end{hypothesis}

Fundamental objects of interest for set-valued mappings are minimal invariant sets.

\begin{definition}[Minimal invariant set]
  A set  $E \in \mathcal{K}(\mathbb{R})$ is called a \emph{minimal invariant} for a set-valued mapping $F: \R \to \cK(\R)$ if $F(E) = E$ and $E$ does
  not contain any proper nonempty subset with this property.
\end{definition}

The importance of minimal invariant sets is due to the fact that they are the supports of stationary densities of the corresponding Markov semi-group \cite{ZmarrouHomburg}.

To study bifurcations of set-valued mappings, we consider families of set-valued dynamical mappings $\{F_\alpha\}_{\alpha\in A}$, where $A \subset \mathbb{R}$ is open.

\begin{hypothesis}[on the family of set-valued mappings $\{F_\alpha\}_{\alpha\in A}$]
  We assume the following hypotheses:
  \begin{itemize}
    \item[(H4)] The mapping $(\alpha, x) \mapsto F_\alpha(x)$ is continuous for all $\alpha\in A$ and $x \in \mathbb{R}$.
    \item[(H5)] The mapping $\alpha \mapsto F_\alpha(x)$ is continuous for every $\alpha\in A$, uniformly in $x\in\R$.
  \end{itemize}
\end{hypothesis}

Let $\mathcal{M}_\alpha$ be the union of all minimal invariant sets of the set-valued mapping $F_\alpha$, $\alpha\in A$. We are interested in discontinuous changes of $\mathcal{M}_\alpha$ under variation of $\alpha$, which we will call \emph{discontinuous topological bifurcations}.

\begin{definition}[Discontinuous topological bifurcation of minimal invariant sets]
  We say that a family of set-valued mappings $\{F_\alpha\}_{\alpha\in A}$ admits a \emph{discontinuous topological bifurcation of minimal invariant sets} at a point $\alpha_0 \in A$ if the mapping
  $\alpha\mapsto \cM_\alpha$ is discontinuous at $\alpha_0$ in Hausdorff distance.
\end{definition}

It will be convenient to define when a single minimal invariant set topologically bifurcates. We say that a family of set-valued mappings $\{F_\alpha\}_{\alpha\in A}$  topologically bifurcates at a minimal invariant set $E_{\alpha_0}$ of $F_{\alpha_0}$ if for any neighborhood $U$ of $E_{\alpha_0}$, the mapping $\alpha\mapsto U\cap \cM_\alpha$ is discontinuous at $\alpha_0$ in Hausdorff distance.

We aim at providing necessary conditions for a discontinuous bifurcation of minimal invariant sets in this paper. As a first step towards this goal, we show that a minimal invariant set satisfying certain properties on the derivatives of the extremal maps persists under continuous perturbations.

\begin{proposition}\label{prop:pers}
  Let $\{F_\alpha \}_{\alpha \in A}$ be a family of set-valued mappings satisfying the assumptions (H1)--(H5). Suppose that there exists a minimal invariant set $E_{\alpha_0} = [e_1, e_2]$ of $F_{\alpha_0}$ such that
 \begin{displaymath}
    \left|\frac{\txtd f^-_{\alpha_0}}{\txtd x}(e_1) \right| < 1\,
    \quad \left|\frac{\txtd f^+_{\alpha_0}}{\txtd x}(e_1) \right| < 1\,,
    \quad \left|\frac{\txtd f^-_{\alpha_0}}{\txtd x}(e_2) \right| < 1\,,
    \quad \left|\frac{\txtd f^+_{\alpha_0}}{\txtd x}(e_2) \right| < 1\,,
  \end{displaymath}
 where $f^-_{\alpha_0}$ and $f^+_{\alpha_0}$ are the extremal mappings of $F_{\alpha_0}$ as defined in (H3). Then there exists a non-bifurcating family of minimal invariant sets in a neighborhood of $\alpha_0$, i.e.~there exists $\delta > 0$ such that for every $\alpha \in B_\delta(\alpha_0)$, there exists a minimal invariant set $E_\alpha$ of $F_\alpha$, and $\{F_\alpha \}_{\alpha \in A}$ does not topologically bifurcate at the minimal invariant set $E_{\alpha_0}$ of $F_{\alpha_0}$.
\end{proposition}

Note that in \cite{ZmarrouHomburg}, a similar result is proved for random diffeomorphisms under the assumption that the minimal invariant set is \emph{isolated}. An isolated minimal invariant set $E$ of a set-valued mapping $F$ satisfies the property that for any open set $W \supset E$, there exists an open set $U$ with $E \subset U \subset W$ such that $F(U) \subset U$. It is easy to show that under the assumptions of Proposition~\ref{prop:pers}, the minimal invariant set $E_{\alpha_0}$ of $F_{\alpha_0}$ is isolated.

\begin{proof}
  We show that there exist two neighborhoods $W_1$ and $W_2$ of the points $e_1$ and $e_2$, respectively, such that for every $x \in W_1 \cup W_2$ one has
  \begin{equation}\label{eq:attractor1}
    \lim_{n\to \infty} d( F_{\alpha_0}^n(x)|E_{\alpha_0} ) = 0\,.
  \end{equation}
  Since $\Big|\frac{\txtd f^-_{\alpha_0}}{\txtd x}(e_1) \Big| < 1$ and $\Big|\frac{\txtd f^+_{\alpha_0}}{\txtd x}(e_1) \Big| < 1$, there exist and $L_1<1$ and a closed neighborhood $W_1$ of $e_1$ such that for every $x,y \in W_1$, we have
  \begin{displaymath}
    |f^-_{\alpha_0} (x) - f^-_{\alpha_0} (y)| \le L_1 | x- y | \quad \mbox{and}\quad  |f^+_{\alpha_0} (x) - f^+_{\alpha_0} (y)| \le L_1 | x- y |\,.
  \end{displaymath}
  Similarly, since $\Big|\frac{\txtd f^-_{\alpha_0}}{\txtd x}(e_2) \Big| < 1$ and $\Big|\frac{\txtd f^+_{\alpha_0}}{\txtd x}(e_2) \Big| < 1$, there exist and $L_2<1$ and a closed neighborhood $W_2$ of $e_2$ such that for every $x,y \in W_2$, we have
  \begin{displaymath}
    |f^-_{\alpha_0} (x) - f^-_{\alpha_0} (y)| \le L_2 | x- y | \quad \mbox{and}\quad  |f^+_{\alpha_0} (x) - f^+_{\alpha_0} (y)| \le L_2 | x- y |\,.
  \end{displaymath}
  With $L:=\max\{L_1, L_2\}<1$, it follows that for every $x \in W_i$, $i\in\{1,2\}$, we have
  \begin{equation}\label{oneitminus}
  | f^-_{\alpha_0} (x) - f^-_{\alpha_0} (e_i) | \le L | x - e_i | = L \dist (x, E_{\alpha_0})\,,
  \end{equation}
  and
  \begin{equation}\label{oneitplus}
  | f^+_{\alpha_0} (x) - f^+_{\alpha_0} (e_i) | \le L | x - e_i | = L \dist (x, E_{\alpha_0})\,.
  \end{equation}
  Define the set $W := W_1 \cup W_2 \cup E_{\alpha_0}$ and note that $F(W)\subset W$. It follows that
  \begin{displaymath}
    d(F^n_{\alpha_0}(x) | E_{\alpha_0}) \le L^n d(W|E_{\alpha_0}) \fa x \in W \mbox{ and } n\in \N\,,
  \end{displaymath}
  where in addition to \eqref{oneitminus} and \eqref{oneitplus}, we also use the invariance of $E_{\alpha_0}$. This establishes \eqref{eq:attractor1}. The conclusion follows using \cite[Proposition~4.1, Theorem~6.1]{LambRasmussenRodrigues}, which makes full use of (H1)--(H5). Any open neighborhood $B_\delta ( E_{\alpha_0})$ with $\delta > 0$ such that $B_\delta ( E_{\alpha_0}) \subset W$ separates the minimal invariant set from its dual according to \cite[Proposition~4.1]{LambRasmussenRodrigues}. By \cite[Theorem~6.1]{LambRasmussenRodrigues}, the necessary condition for a discontinuous topological bifurcation is not satisfied, which means that of $F_{\alpha_0}$ does not topologically bifurcate at the minimal invariant set $E_{\alpha_0}$. We note that the results in \cite{LambRasmussenRodrigues} are formulated for compact spaces, but the results can be applied if our family of set-valued mappings is truncated outside a compact set that contains the minimal invariant set $E_{\alpha_0}$ in its interior.
\end{proof}

It is possible to formulate a necessary condition for having a bifurcation that is a direct consequence of Proposition \ref{prop:pers} and is expressed in terms of the derivatives of the extremal maps.

\begin{corollary}[Necessary condition for a discontinuous topological bifurcation]\label{cor:necessary}
  Let $\{F_\alpha\}_{\alpha \in A}$ be a family of set-valued mappings satisfying (H1)--(H5). For a fixed $\alpha_0\in A$, let $E_{\alpha_0}=[e_1,e_2]$ be a minimal invariant set of $F_{\alpha_0}$. If $\{F_\alpha\}_{\alpha \in A}$ topologically bifurcates at $E_{\alpha_0}$, then at least one of the four conditions
  \[
   \left| \frac{\txtd f^{-}_{\alpha_0}}{\txtd x}(e_1) \right| \ge 1\,,
  \quad \left| \frac{\txtd f^+_{\alpha_0}}{\txtd x}(e_1) \right| \ge 1\,
,  \quad \left| \frac{\txtd f^{-}_{\alpha_0}}{\txtd x}(e_2) \right| \ge 1\,,
  \quad \left| \frac{\txtd f^+_{\alpha_0}}{\txtd x}(e_2) \right| \ge 1
  \]
  is satisfied.
\end{corollary}

The concept of collision between minimal invariant sets and their duals as presented in \cite{LambRasmussenRodrigues}, and that is used in the proof of Proposition~\ref{prop:pers} is very powerful, since it allows to ignore the interior of the minimal invariant set in the bifurcation analysis and focus the study on the dynamics close to its boundary. We will use this observation in Section~\ref{sec:algorithm} to develop early-warning signals for upcoming bifurcations.

We now present two examples of bifurcations in set-valued mappings.

\begin{example}\label{ex:pitch}
Consider the family of set-valued mappings $F_\alpha:\R\to\cK(\R)$, $\alpha >0$, given by
\[
F_\alpha(x) := \overline{B_\sigma \left( \tfrac{1}{2} \alpha \arctan(x) + \tfrac{1}{2} x \right)} \fa x\in\R\,,
\]
where $\sigma > 0 $. For a given and fixed $\sigma > 0$, there exists $\alpha_0>0$ such that the following statements hold, describing a set-valued pitchfork bifurcation.
\begin{itemize}
\item For $\alpha \ge \alpha_0$, there exist exactly two minimal invariant sets $E_\alpha^1 \subset (-\infty,0)$ and $E_\alpha^2\subset (0,\infty)$. They are symmetric with respect to the origin, since $F_\alpha$ is symmetric as well. For $\alpha > \alpha_0$, each extremal map $f_\alpha^+$ and $f_\alpha^-$ of the set-valued mapping $F_\alpha$ has a pair of attracting fixed points, which are boundary points of the minimal invariant sets $E_\alpha^1$ and $E_\alpha^2$.
\item   For $\alpha < \alpha_0$, there exists only one minimal invariant set $E_\alpha$, that is not a continuation of the two minimal invariant sets from above, since it also contains the interval between these two sets. Each extremal mapping $f_\alpha^+$ and $f_\alpha^-$ now has only one fixed point that bounds the new minimal invariant set (see Figure~\ref{figure:pitchbifdiagram}).
\end{itemize}
\begin{figure}[htbp]
  \begin{center}
   \begin{overpic}[width=7cm,unit=1mm]{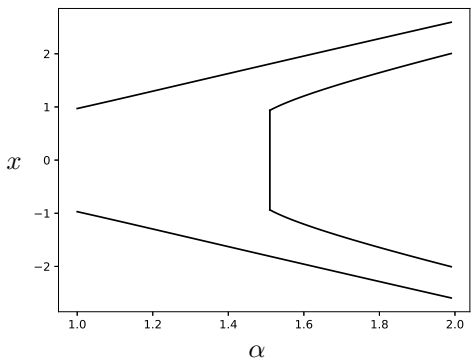}
\put(10,0){{\parbox{0.4\linewidth}{%
     \[
       \alpha
     \]}}}
\put(-22,25){{\parbox{0.4\linewidth}{%
     \[
        x
     \]}}}
\end{overpic}
    \caption{\label{figure:pitchbifdiagram} Dependence of the minimal invariant sets on $\alpha$ for the family of set-valued mappings $F_\alpha$, $\alpha>0$, from Example~\ref{ex:pitch}: there are two minimal invariant sets for $\alpha>\alpha_0$, and after the discontinuous bifurcation, there is only one bigger minimal invariant set for $\alpha\le \alpha_0$.}
  \end{center}
\end{figure}
Importantly, the value $\alpha_0$ depends on $\sigma$ and can be observed as the value such that one of the extremal mappings $f_\alpha^+$ or $f_\alpha^-$ is tangent to the identity line (see Figure~\ref{figure:pitchfork}), meaning the derivative in this intersection point is equal to $1$. We will use this later to develop an early-warning signal for this bifurcation.
\begin{figure}[htbp]
  \begin{center}
  \begin{minipage}{0.3\textwidth}
 \begin{overpic}[width=\textwidth,unit=1mm]{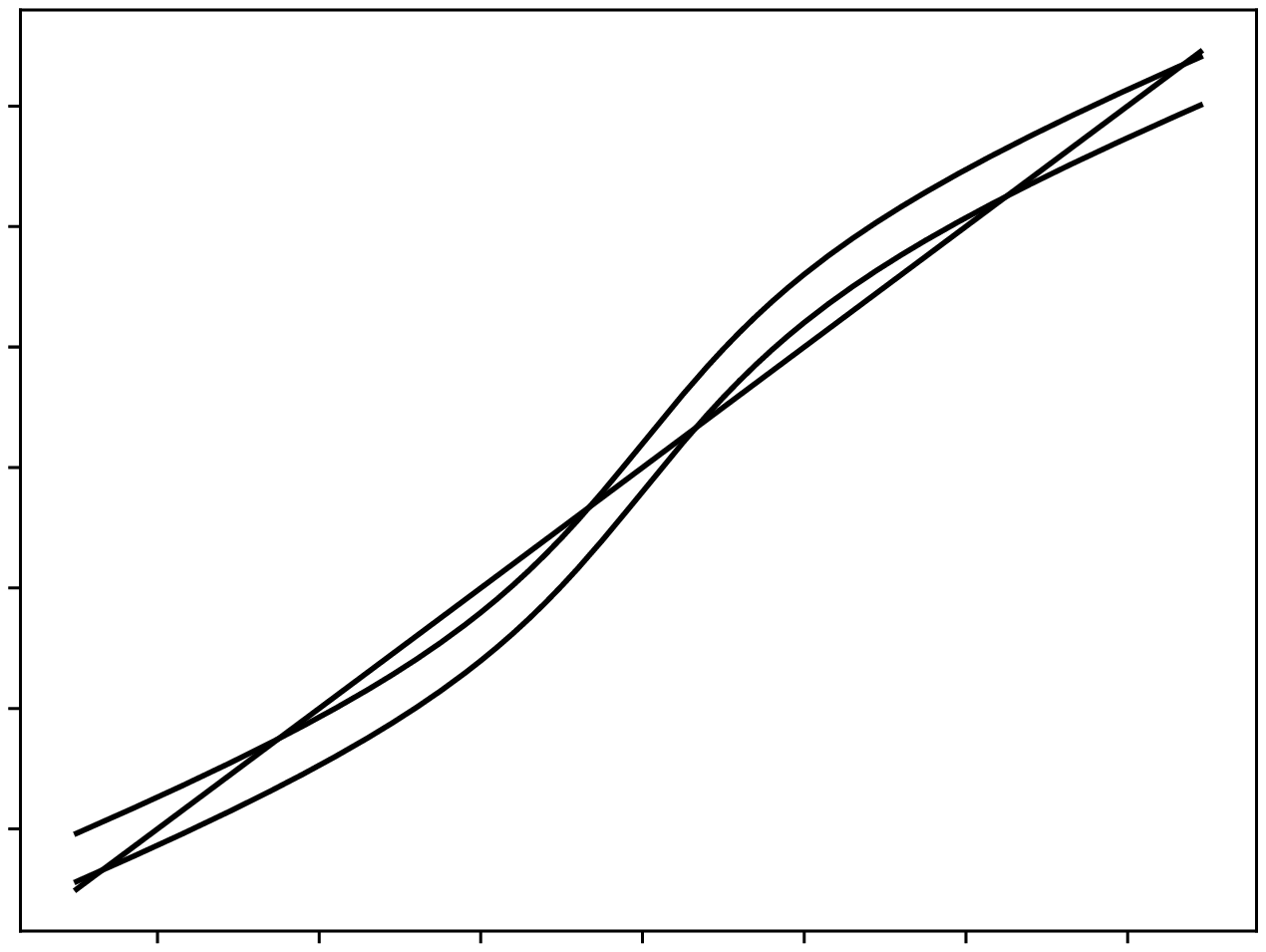}
\put(12,-3){{\parbox{0.4\linewidth}{%
     \[
       x
     \]}}}
\put(12,1){{\parbox{0.4\linewidth}{%
     \[
       0
     \]}}}
\put(-2,0){{\parbox{0.5\linewidth}{%
     \[
       -3
     \]}}}
\put(21.5,0){{\parbox{0.5\linewidth}{%
     \[
       3
     \]}}}

\put(-8.5,4){{\parbox{0.5\linewidth}{%
     \[
       -3
     \]}}}
\put(-7,13){{\parbox{0.5\linewidth}{%
     \[
       0
     \]}}}
\put(-7,21.5){{\parbox{0.5\linewidth}{%
     \[
       3
     \]}}}
\end{overpic}
  \end{minipage}
  \begin{minipage}{0.3\textwidth}
 \begin{overpic}[width=\textwidth,unit=1mm]{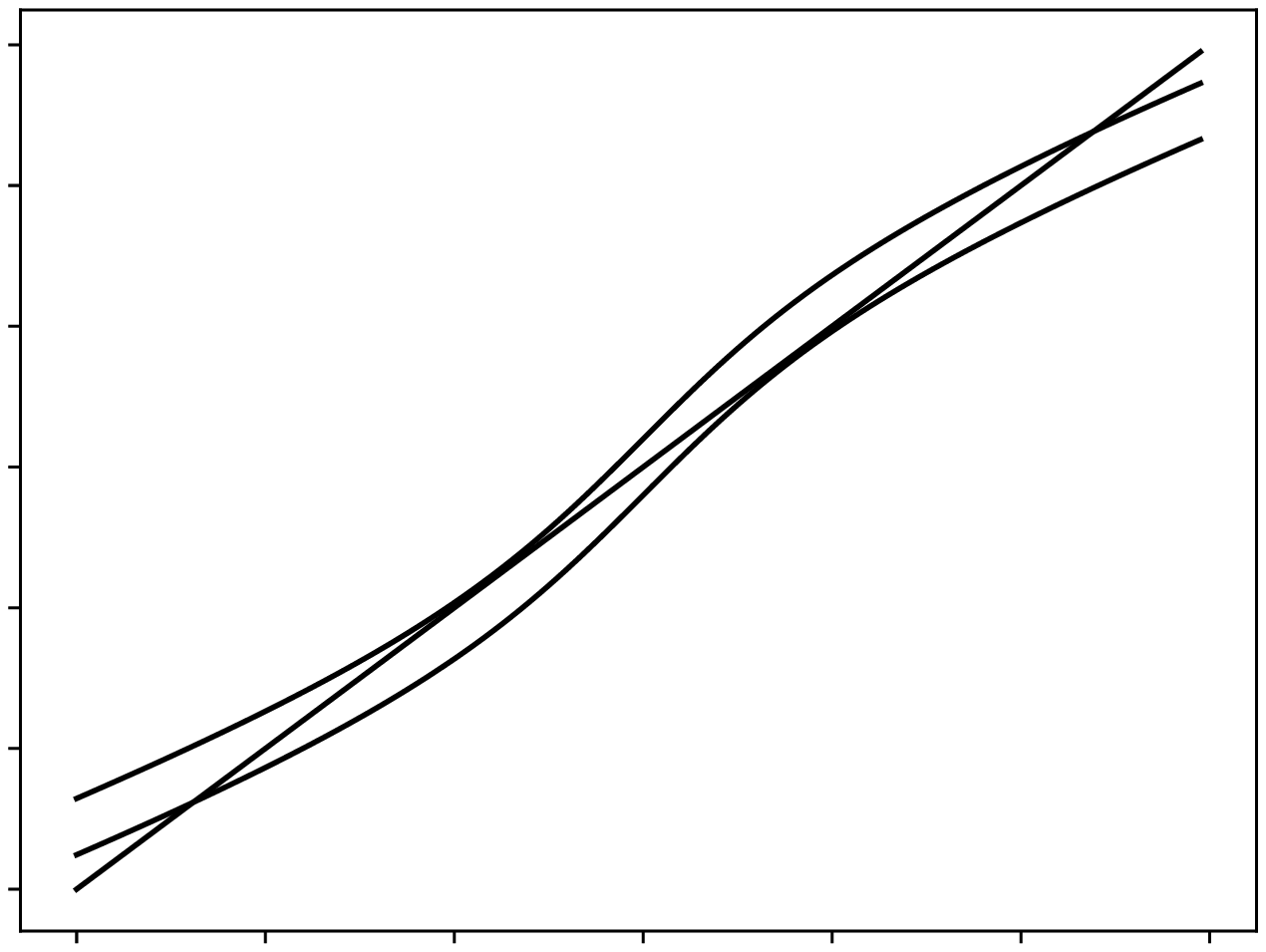}
\put(12,-3){{\parbox{0.4\linewidth}{%
     \[
       x
     \]}}}
\put(12,1){{\parbox{0.4\linewidth}{%
     \[
       0
     \]}}}
\put(-4,0){{\parbox{0.5\linewidth}{%
     \[
       -3
     \]}}}
\put(23.2,0){{\parbox{0.5\linewidth}{%
     \[
       3
     \]}}}

\put(-8.5,2.5){{\parbox{0.5\linewidth}{%
     \[
       -3
     \]}}}
\put(-7,13){{\parbox{0.5\linewidth}{%
     \[
       0
     \]}}}
\put(-7,23){{\parbox{0.5\linewidth}{%
     \[
       3
     \]}}}
\end{overpic}
  \end{minipage}
    \begin{minipage}{0.3\textwidth}
 \begin{overpic}[width=\textwidth,unit=1mm]{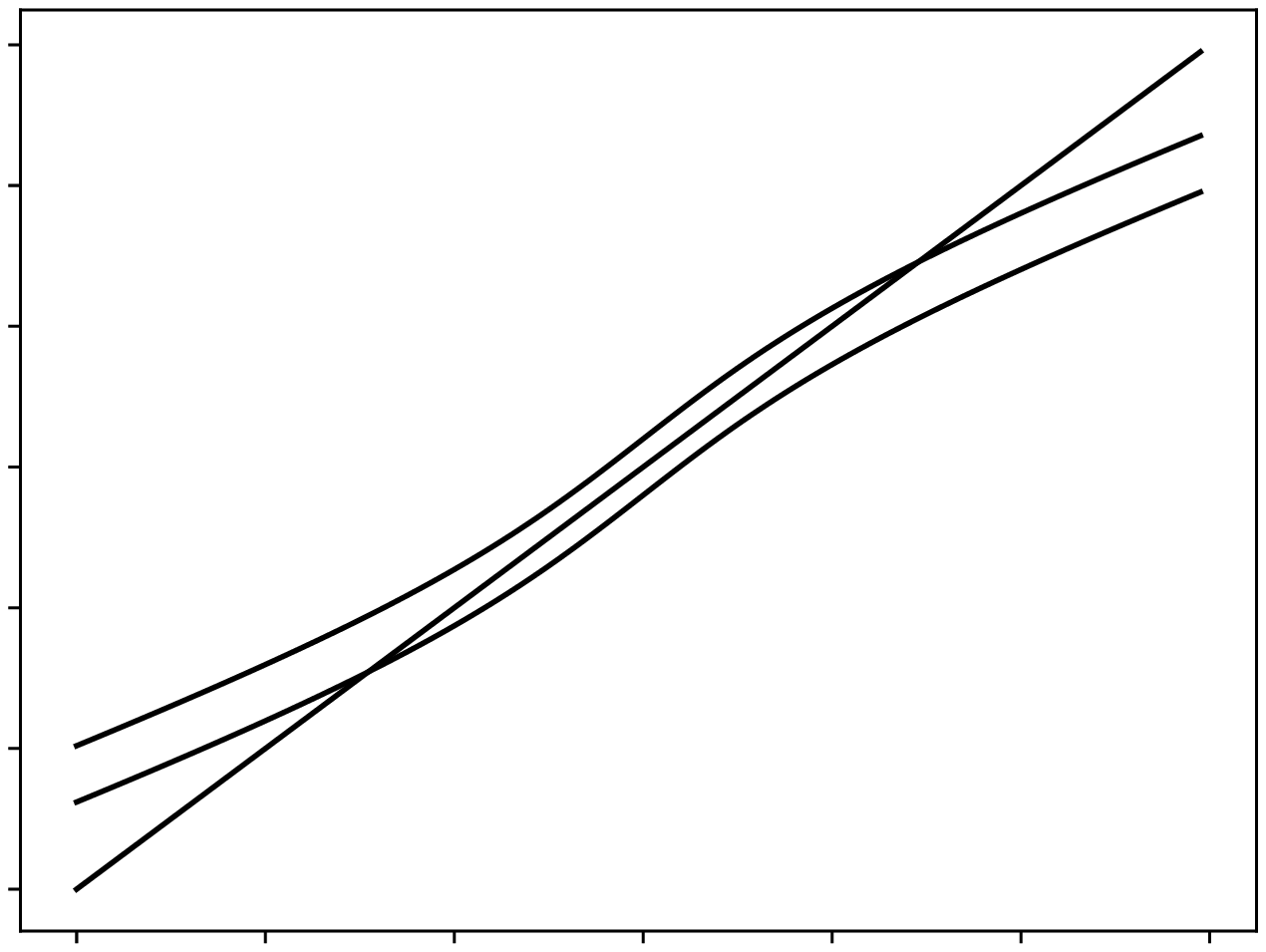}
\put(12,-3){{\parbox{0.4\linewidth}{%
     \[
       x
     \]}}}
\put(12,1){{\parbox{0.4\linewidth}{%
     \[
       0
     \]}}}
\put(-4,0){{\parbox{0.5\linewidth}{%
     \[
       -3
     \]}}}
\put(23.2,0){{\parbox{0.5\linewidth}{%
     \[
       3
     \]}}}

\put(-8.5,2.5){{\parbox{0.5\linewidth}{%
     \[
       -3
     \]}}}
\put(-7,13){{\parbox{0.5\linewidth}{%
     \[
       0
     \]}}}
\put(-7,23){{\parbox{0.5\linewidth}{%
     \[
       3
     \]}}}
\end{overpic}
  \end{minipage}
  \end{center}
  \caption{\label{figure:pitchfork} Profiles of the extremal maps $f_\alpha^+(x)$ and $f_\alpha^-(x)$ from Example~\ref{ex:pitch}. From left to right: $\alpha >\alpha_0$ (two minimal invariants sets), $\alpha = \alpha_0$ (two minimal invariant sets, $\alpha <\alpha_0$ (one minimal invariant set). }
\end{figure}

Note that this bifurcation can not be detected in the limit $\alpha\nearrow \alpha_0$ by considering derivatives of the extremal maps at the boundary of $E_\alpha$, $\alpha<\alpha_0$. This does not contradict the statement of
Proposition~\ref{prop:pers} and Corollary~\ref{cor:necessary}, since at $\alpha=\alpha_0$, there are two minimal invariant sets, and the necessary condition for a bifurcation established in Corollary~\ref{cor:necessary} holds for both of these invariant sets.
\end{example}

In the following proposition we provide a sufficient condition for a discontinuous topological bifurcation. We require that the extremal mappings are strictly monotone increasing.

\begin{proposition}[Sufficient condition for a discontinuous topological bifurcation]\label{prop:sufficient}
  Consider open intervals $I, U\subset\R$ and family of set-valued mappings $F_\alpha:I\to\cK(I)$, $\alpha \in U$. We assume that the extreme mapping $f^+_\alpha$ as defined in \eqref{botmap} is two times continuously differentiable. In addition, we require that $\alpha \mapsto f^+_\alpha(x)$ is differentiable for all $x \in \mathbb{R}$. Assume that for a fixed $\alpha_0\in U$, there exists a minimal invariant set $E_{\alpha_0} = [e_1(\alpha_0), e_2(\alpha_0)]$. Furthermore, we suppose that the following four conditions hold.
  \begin{enumerate}
    \item[(i)] $f^+_\alpha: I \to I$ is strictly monotone increasing for every $\alpha \in U$.
    \item[(ii)] $\frac{\txtd}{\txtd x}  f^+_{\alpha_0}(x) \big|_{x = e_2(\alpha_0)}  =1$.
    \item[(iii)] $\frac{\txtd^2 }{\txtd x^2} f^+_{\alpha_0}(x) \big|_{x = e_2(\alpha_0)} > 0$.
    \item[(iv)] $\frac{\txtd }{\txtd\alpha} f^+_{\alpha}(e_2(\alpha_0)) \big|_{\alpha = \alpha_0} > 0$.
  \end{enumerate}
  Then the family of set-valued mappings $(F_\alpha)_{\alpha\in U}$ admits a discontinuous topological bifurcations at the minimal invariant set $E_{\alpha_0}$.
\end{proposition}

\begin{remark}
  Analogous conditions can be formulated for $e_1(\alpha_0) \in E_{\alpha_0}$ and the extremal mappings $f^-_{\alpha}$, $\alpha\in U$. In this case we would then require that
  $\frac{\txtd^2 }{\txtd x^2}f^-_{\alpha}(x)\big|_{x = e_1(\alpha_0)}<0 $ as well as $\frac{\txtd }{\txtd\alpha}f^-_{\alpha}(e_1(\alpha_0))\big|_{\alpha = \alpha_0} < 0$.
\end{remark}

\begin{proof}[Proof of Proposition~\ref{prop:sufficient}]
  Assume that the family $E_{\alpha}$ does not admit a discontinuous topological bifurcation at $\alpha_0$. Then there exists a neighborhood $W\subset U$ of $\alpha_0$ and a continuous family of minimal invariant sets $E_\alpha=[e_1(\alpha), e_2(\alpha)]$ for $F_\alpha$, $\alpha\in W$.  Because of (i) and the minimal invariance of $E_{\alpha}=[e_1(\alpha), e_2(\alpha)]$ for $F_{\alpha}$, it follows that
  \begin{equation}\label{eq:limbound}
    f^+_{\alpha}(e_2(\alpha)) = e_2(\alpha) \fa \alpha\in W\,.
  \end{equation}
  Conditions (ii), (iii), (iv) imply that the mappings $f_\alpha^+$, $\alpha\in W$, admit a saddle node bifurcation at $(\alpha_0,e_2(\alpha_0))$, which means in particular, that there exists an $\bar{\alpha}> \alpha_0$ and a neighborhood $V\subset I$ of $x_0$ such that the mapping $f_\alpha^+$ has no fixed points in $V$ for any $\alpha\in(\alpha_0,\bar \alpha)$. This contradicts \eqref{eq:limbound} and finishes the proof of this proposition.
\end{proof}

We now consider a topological bifurcation for a non-invertible mapping.

\begin{example}
\label{ex:doubling}
Consider the quadratic map $f_\alpha:\R\to\R$, $f_\alpha(x) = - x^2 + \alpha$.  We consider the corresponding set-valued mapping $F_\alpha:\R\to\cK(\R)$, given by
\[
F_\alpha(x) :=\overline{B_\sigma(  - x^2 + \alpha )}\,,
\]
for a fixed noise level $\sigma>0$. In the following, we determine the values of $\sigma > 0$ and $\alpha>0$ such that a single connected minimal invariant set transitions discontinuously into a minimal invariant set that is the union of two intervals that are periodically mapped into each other. We briefly outline how these two different types of minimal invariant sets (i.e.~the single interval before the bifurcation and the two intervals after the bifurcation) are found for this non-invertible mapping, analogously to considering the fixed points of the extremal maps in the order-preserving and invertible case studied in Example~\ref{ex:pitch}.

Consider first an $\alpha>0$ for which there exists a minimal invariant set $E_\alpha = [a,b]$ of $F_\alpha(x)$ such that $f_\alpha^\pm$ are strictly monotonically decreasing on $E_\alpha$. Note that $f^+_\alpha (a) = b$ and $f^-_\alpha ( b ) = a$, since $f^+_\alpha$ is order-reversing. We get
\[
a = f_\alpha^+ (f_\alpha^- ( a)) \quad \text{and} \quad b = f_\alpha^- (f_\alpha^+ ( b))\,.
\]

Consider now an $\alpha>0$  for which there exists a minimal invariant set $E_\alpha = E^1_\alpha \cup E^2_\alpha = [a,c] \cup [d,b]$ such that $f_\alpha^\pm$ are strictly monotonically decreasing on $E_\alpha$.  Since $F_\alpha(E^1_\alpha) = E^2_\alpha$ and  $F_\alpha(E^2_\alpha) = E^1_\alpha$, we have
\[
f_\alpha^+(a) = b\,, \quad f^-_\alpha(b) = a\,, \quad f^-_\alpha(c) = d  \quad \text{and}\quad f^+_\alpha(d) = c\,.
\]
It follows that
\[
a = f_\alpha^- (f_\alpha^+ ( a)), \quad
b = f_\alpha^+ (f_\alpha^- ( b)), \quad
c = f_\alpha^+ (f_\alpha^- ( c)), \quad d = f_\alpha^- (f_\alpha^+ ( d)).
\]

It follows in both cases that the boundary points of $E$ are fixed points of $f_\alpha^+\circ f_\alpha^-$ or $f_\alpha^-\circ f_\alpha^+$, so these mappings play a crucial role in the bifurcation analysis.
We describe the bifurcation occurring for $F_\alpha$ as illustrated in Figure~\ref{figure:diagrdoubl} for a minimal invariant set located in $[0, \infty)$. For a choice of a fixed and small enough $\sigma > 0$, there exist $\alpha_0 > 0$ and a neighborhood $B_\delta(\alpha_0)$, $\delta> 0$, of $\alpha_0$, both depending on $\sigma$, such that the following statements hold.

\begin{figure}[htbp]
 \begin{center}
   \begin{overpic}[width=7cm,unit=1mm]{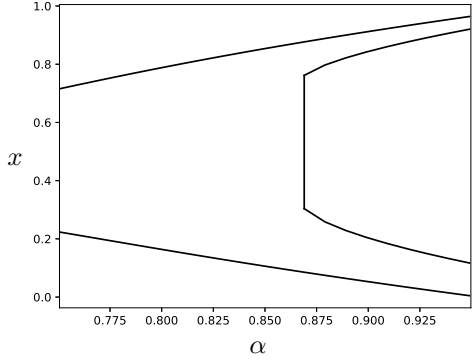}
\put(10,0){{\parbox{0.4\linewidth}{%
     \[
       \alpha
     \]}}}
\put(-22,25){{\parbox{0.4\linewidth}{%
     \[
        x
     \]}}}
\end{overpic}

  \caption{\label{figure:diagrdoubl} Dependence of the minimal invariant sets on $\alpha$ for the family of set-valued mappings $F_\alpha$ from Example~\ref{ex:doubling} with $\sigma = 0.015$: there is a connected minimal invariant set for $\alpha<\alpha_0$, and after the discontinuous bifurcation, there exists a disconnected minimal invariant given as union of two intervals for $\alpha \ge \alpha_0$.}
 \end{center}
\end{figure}

\begin{figure}[htbp]
  \begin{center}
  \begin{minipage}{0.3\textwidth}
  \begin{overpic}[width=\textwidth,unit=1mm]{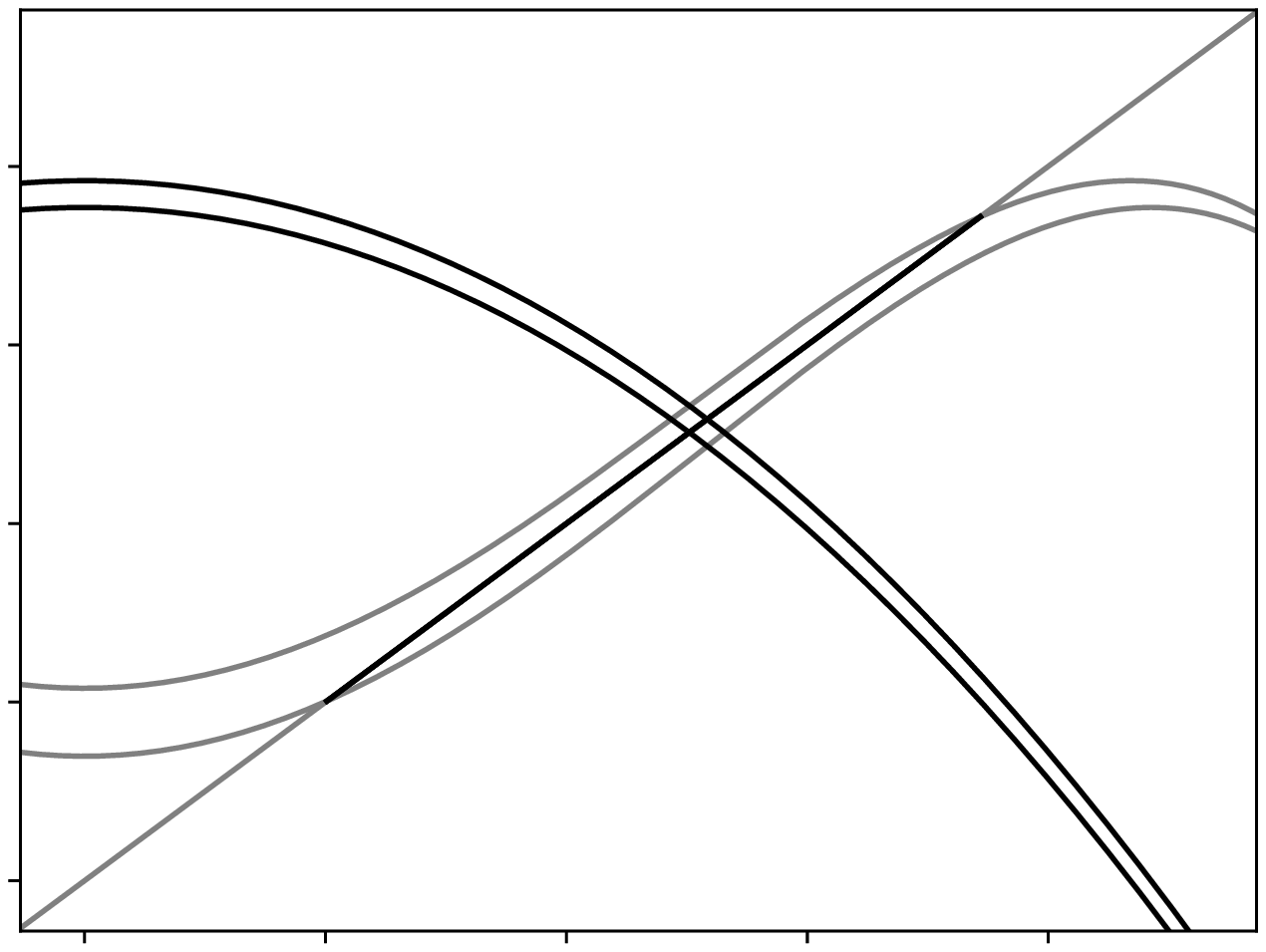}
\put(12,-3){{\parbox{0.4\linewidth}{%
     \[
       x
     \]}}}
\put(10,1){{\parbox{0.4\linewidth}{%
     \[
       0.4
     \]}}}
\put(-3,0){{\parbox{0.5\linewidth}{%
     \[
       0
     \]}}}
\put(19.5,0){{\parbox{0.5\linewidth}{%
     \[
       0.8
     \]}}}
\put(-6.5,3.5){{\parbox{0.5\linewidth}{%
     \[
       0
     \]}}}
\put(-7.5,12){{\parbox{0.5\linewidth}{%
     \[
       0.4
     \]}}}
\put(-7.5,20){{\parbox{0.5\linewidth}{%
     \[
       0.8
     \]}}}
\end{overpic}
  \end{minipage}
  \begin{minipage}{0.3\textwidth}
  \begin{overpic}[width=\textwidth,unit=1mm]{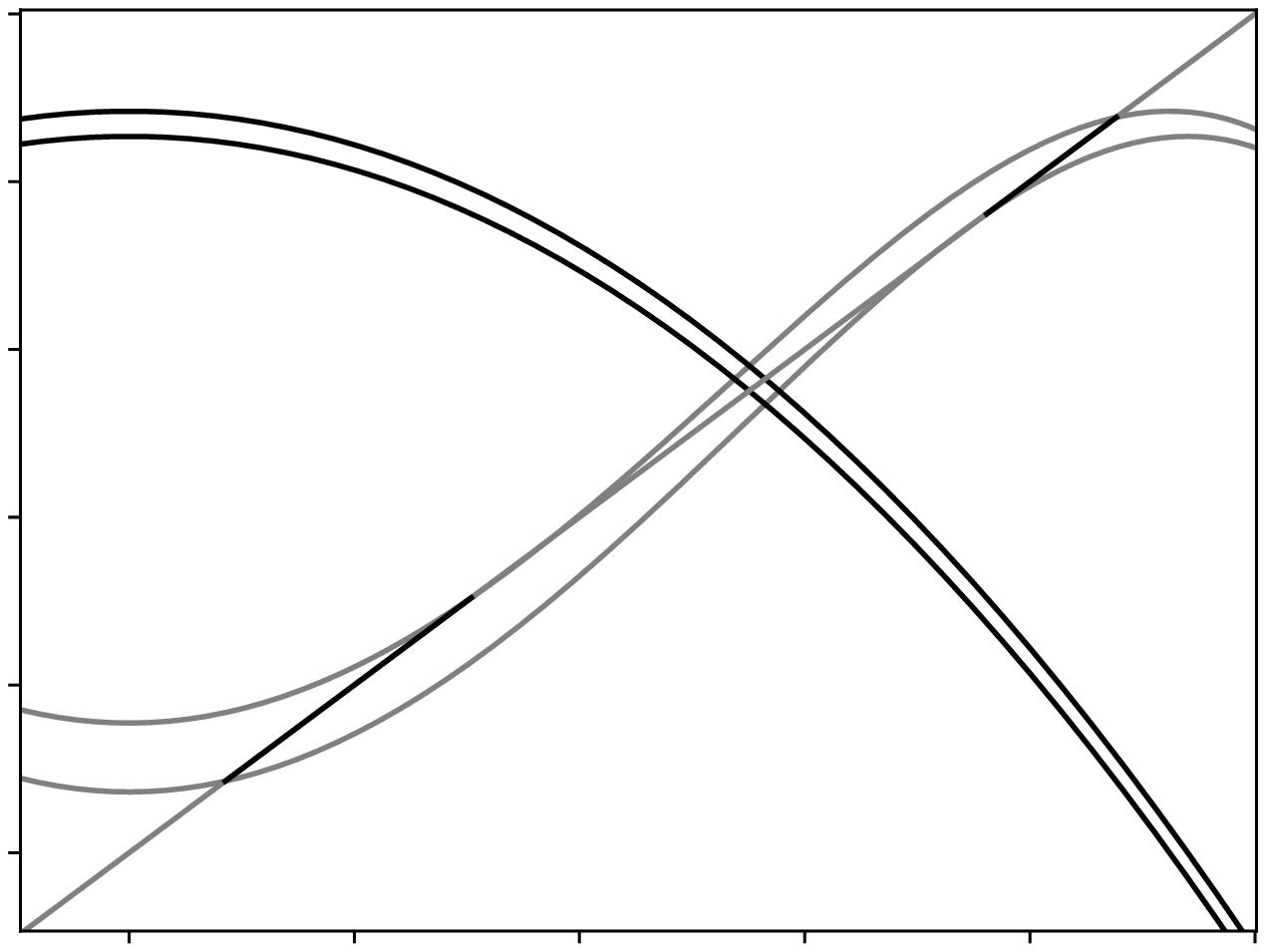}
\put(12,-3){{\parbox{0.4\linewidth}{%
     \[
       x
     \]}}}
\put(10.2,1){{\parbox{0.4\linewidth}{%
     \[
       0.4
     \]}}}
\put(-2,0){{\parbox{0.5\linewidth}{%
     \[
       0
     \]}}}
\put(19,0){{\parbox{0.5\linewidth}{%
     \[
       0.8
     \]}}}
\put(-6.5,4.1){{\parbox{0.5\linewidth}{%
     \[
       0
     \]}}}
\put(-7.5,12){{\parbox{0.5\linewidth}{%
     \[
       0.4
     \]}}}
\put(-7.5,19.5){{\parbox{0.5\linewidth}{%
     \[
       0.8
     \]}}}
\end{overpic}
  \end{minipage}
    \begin{minipage}{0.3\textwidth}
 \begin{overpic}[width=\textwidth,unit=1mm]{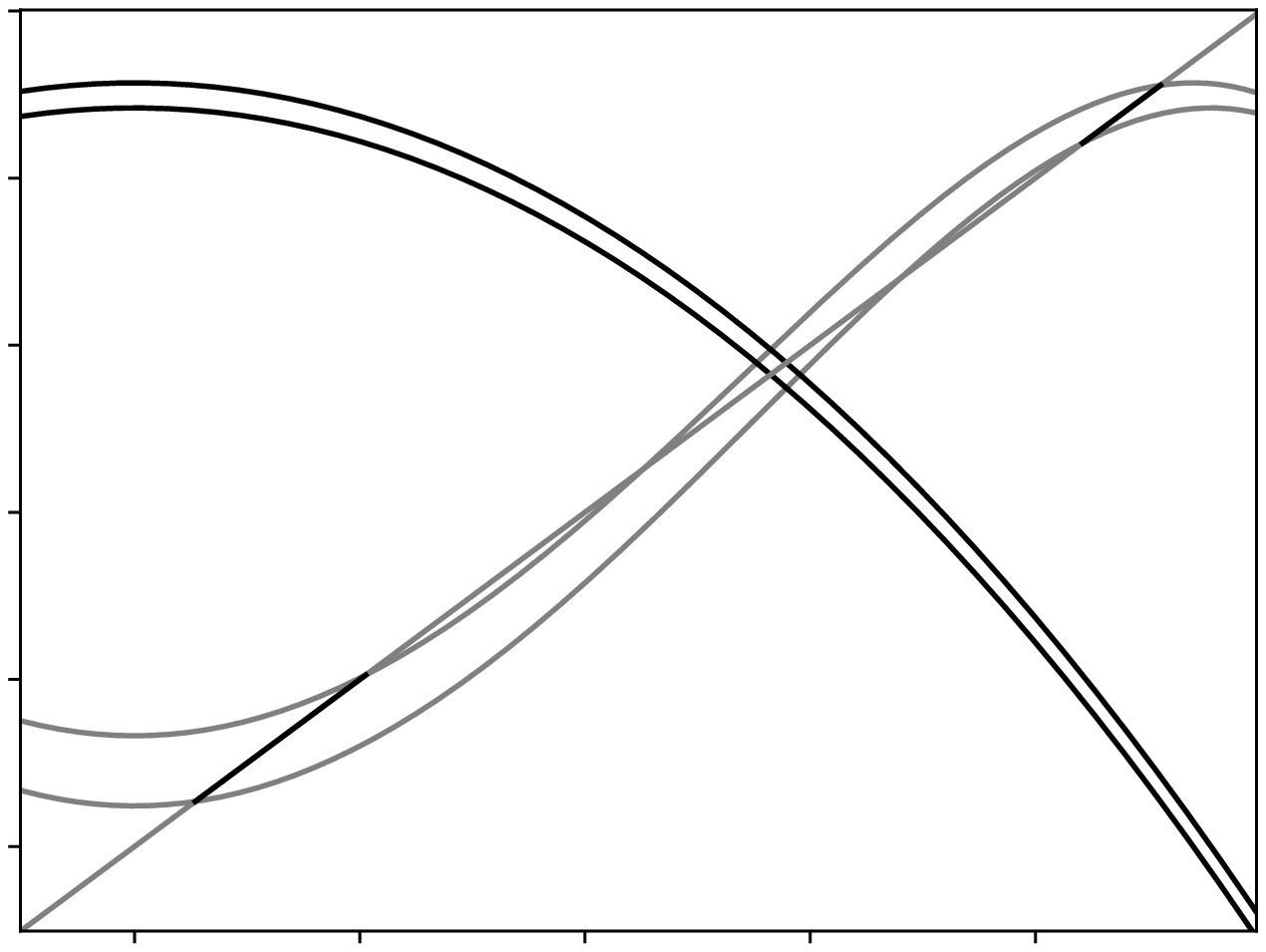}
\put(12,-3){{\parbox{0.4\linewidth}{%
     \[
       x
     \]}}}
\put(10.2,1){{\parbox{0.4\linewidth}{%
     \[
       0.4
     \]}}}
\put(-2,0){{\parbox{0.5\linewidth}{%
     \[
       0
     \]}}}
\put(19.2,0){{\parbox{0.5\linewidth}{%
     \[
       0.8
     \]}}}
\put(-6.5,4.1){{\parbox{0.5\linewidth}{%
     \[
       0
     \]}}}
\put(-7.5,12){{\parbox{0.5\linewidth}{%
     \[
       0.4
     \]}}}
\put(-7.5,20){{\parbox{0.5\linewidth}{%
     \[
       0.8
     \]}}}
\end{overpic}
  \end{minipage}
  \end{center}

  \caption{\label{figure:doubling} The extremal maps $f_\alpha^+$ and $f_\alpha^+$ (in black), $f_\alpha^- \circ f_\alpha^+$ and $f_\alpha^+ \circ f_\alpha^-$ (in gray) from Example~\ref{ex:doubling} with $\sigma = 0.015$. The minimal invariant sets are plotted in black on the identity line for convenience. \emph{Left:} $\alpha <\alpha_0$; one connected minimal invariant set. \emph{Middle:} $\alpha = \alpha_0$; one minimal invariant set with two connected components; the second iteration mappings are tangent to the identity. \emph{Right:} $\alpha > \alpha_0$; one minimal invariant set with two connected components. }
\end{figure}

\begin{itemize}
  \item[(i)]
    For $\alpha \in (\alpha_0-\delta, \alpha_0)$, there exists a single connected minimal invariant set $E_\alpha$. The map $f_\alpha^- \circ f_\alpha^+$ (bottom map in gray in Figure~\ref{figure:doubling}) has one (attractive) fixed points $q^-(\alpha)$ and the map $f_\alpha^+ \circ f_\alpha^-$ (top map in gray in Figure~\ref{figure:doubling}) has also one (attractive) fixed point $q^+(\alpha)$, with $q^-(\alpha) <  q^+(\alpha)$. Due to the above analysis, we obtain
    \[
    E_\alpha = [q^-(\alpha), q^+(\alpha)]\fa \alpha \in (\alpha_0-\delta, \alpha_0)\,.
    \]
    Note that the fixed points $q^-(\alpha)$ and $q^+(\alpha)$ can be continued on the interval $(\alpha_0-\delta, \alpha_0+\delta)$.
  \item[(ii)]
    At $\alpha = \alpha_0$, both maps $f_\alpha^- \circ f_\alpha^+$ and $f_\alpha^+ \circ f_\alpha^- $ undergo a saddle node bifurcation, and for each mapping, a new pair of fixed points is created: $r_1^-(\alpha)$ and $r_2^-(\alpha)$ for the mapping $f_\alpha^- \circ f_\alpha^+$ (note that $r_1^-(\alpha_0) = r_2^-(\alpha_0)$), and $r_1^+(\alpha)$ and $r_2^+(\alpha)$ for the mapping $f_\alpha^+ \circ f_\alpha^-$  (note that $r_1^+(\alpha_0) = r_2^+(\alpha_0)$). Both pairs are defined for $\alpha\in [\alpha_0, \alpha_0+\delta)$.\\
    For $\alpha \in [\alpha_0, \alpha_0 + \delta)$, the four fixed points of the second iterates $q^-(\alpha) < r_1^+(\alpha) < r_2^-(\alpha) < q^+(\alpha)$ are attracting and two minimal invariant sets for the second iterates exist, given by $E^1_\alpha = [q^-(\alpha),r_1^+(\alpha)]$ and $E_\alpha^2 = [r_2^-(\alpha),q^+(\alpha)]$. The set $[q^-(\alpha), q^+(\alpha)]$ is still invariant for $F_\alpha$, but it is not minimal.  We have
    \[
     F_\alpha( E_\alpha^1  ) =  E_\alpha^2 \textrm{ and } F_\alpha ( E_\alpha^2 ) = E_\alpha^1\,,
    \]
    and the minimal invariant set is given by $E_\alpha = E_\alpha^1 \cup E_\alpha^2$.
\end{itemize}
The choice of $\sigma$ is important, since for larger values of $\sigma$, the minimal invariant set may contain the critical point at $x=0$, and the boundary maps are then no longer order-reversing on the minimal invariant set. We choose $\sigma$ small enough, so that this does not happen. If the critical point is not contained in the invariant set, then the extremal mappings are strictly monotonically decreasing on the minimal invariant set, and the previous statements hold.

The bifurcation point $\alpha_0$ depends only on $\sigma$ and it is calculated as follows. Consider $f_\alpha^- \circ f_\alpha^+$ or $f_\alpha^+ \circ f_\alpha^-$, then $\alpha_0$ is the parameter value such that
\begin{equation}\label{rel1}
f_{\alpha}^- (f_{\alpha}^+ ( x)) = x \quad \text{and}\quad f_{\alpha}^+ (f_{\alpha}^- ( x))  = x
\end{equation}
for $x > 0$ has exactly two solutions, one of which is a tangential to the identity line. The crucial point here is that this happens for both equations in \eqref{rel1} at the same~$\alpha_0$. In fact, assume that the map $f_{\alpha_0}^+ \circ f_{\alpha_0}^-$ has only one fixed point $q^+(\alpha_0)>0$, and the map $f_{\alpha_0}^- \circ f_{\alpha_0}^+$ has two fixed points $0< q^-(\alpha_0) <  r_1^-(\alpha_0) = r_2^-(\alpha_0)$. Then define the set $E^2_{\alpha_0} := [r_2^-(\alpha_0),q^+(\alpha_0)]$. Since the extremal mappings $f^+_{\alpha_0}$ and $f^-_{\alpha_0}$ are strictly monotonically decreasing on $[q^-(\alpha_0),q^+(\alpha_0)]$, we have
\begin{equation}\label{eq:cycle1}
  E^1_{\alpha_0} := F_\alpha(E^2_{\alpha_0}) = [f_{\alpha_0}^-(q^+(\alpha_0)), f_{\alpha_0}^+(r_2^-(\alpha_0))]\,.
\end{equation}
It is easy to see that $E^1_{\alpha_0} \cap E^2_{\alpha_0} = \emptyset$. We iterate again and use the fact that $q^+(\alpha_0)$ and $r_2^-(\alpha_0)$ are fixed points of $f_{\alpha_0}^+ \circ f_{\alpha_0}^-$ and $f_{\alpha_0}^- \circ f_{\alpha_0}^+$, respectively. We obtain
\begin{equation}
\label{eq:cycle2}
F_\alpha(E^1_{\alpha_0}) = [f_{\alpha_0}^- ( f_{\alpha_0}^+(r_2^-(\alpha_0))),f_{\alpha_0}^+ (f_{\alpha_0}^-(q^+(\alpha_0)))] = [r_2^-(\alpha_0),q^+(\alpha_0)] =  E^1_{\alpha_0}\,.
\end{equation}
The sets $E^1_{\alpha_0}$ and $E^2_{\alpha_0}$ are cyclically mapped into each other by $F_\alpha$ and by combining \eqref{eq:cycle1} and \eqref{eq:cycle2} we have
\begin{align*}
& [f_{\alpha_0}^-(q^+(\alpha_0)), f_{\alpha_0}^+(r_2^-(\alpha_0))] = E^2_{\alpha_0} = \\
  = & F_\alpha (E^1_{\alpha_0}) = [ f_{\alpha_0}^-( f_{\alpha_0}^+ (f_{\alpha_0}^-(q^+(\alpha_0))) ), f_{\alpha_0}^+( f_{\alpha_0}^- ( f_{\alpha_0}^+(r_2^-(\alpha_0)))) ]\,.
\end{align*}
Then one has
\[
  f_{\alpha_0}^+( f_{\alpha_0}^- ( f_{\alpha_0}^+(r_2^-(\alpha_0)))) = f_{\alpha_0}^+(r_2^-(\alpha_0))\,,
\]
i.e.~$f_{\alpha_0}^+(r_2^-(\alpha_0))$ is a fixed point of $f_{\alpha_0}^+ \circ f_{\alpha_0}^-$. It is straightforward to calculate that $f_{\alpha_0}^+(r_2^-(\alpha_0)) > q^-(\alpha_0) $ and $f_{\alpha_0}^+(r_2^-(\alpha_0)) < r_2^-(\alpha_0)$. This contradicts the assumption that $f_{\alpha_0}^+ \circ f_{\alpha_0}^-$ has only one fixed point for $x > 0$. It follows that $f_{\alpha_0}^- \circ f_{\alpha_0}^+$ and $f_{\alpha_0}^+ \circ f_{\alpha_0}^-$ both have two exactly different fixed points at $\alpha = \alpha_0$ and $x > 0$, one of which is tangent to the identity.

We note that the extremal maps at the boundary of the sets still satisfy the necessary condition of Corollary~\ref{cor:necessary} (for any $x > \frac{1}{2}$, $x \in E_\alpha$, one has $|(f^\pm_\alpha)^\prime (x)| > 1$),  but the derivative of the extremal maps at the boundary is not a good indicator for this bifurcation, since the bifurcation is not caused by an exchange of stability at the boundary of one of the extremal maps. The compositions of the extremal maps $f_\alpha^- \circ f_\alpha^+$ and $f_\alpha^+ \circ f_\alpha^-$ must be considered instead here, although the bifurcation diagram is very similar to the the pitchfork bifurcation described in Example~\ref{ex:pitch}.
\end{example}

%%%%%%%%%%%%%%%%%%%%%%%%%%%%%%%%%%%%%%%%%%%%%%%%%%%%%%%%%%%%%%%%%%%%%%%%%%%%%%%%%%%%%%%%%%%%
\section{Early-warning signals for bifurcations in random dynamical systems with bounded noise}
\label{sec:algorithm}

Set-valued mappings describe the topological part of a discrete-time random dynamical system with bounded noise, in the sense that the support of the ergodic stationary measures correspond
to the minimal invariant sets of the set-valued mapping \cite{ZmarrouHomburg}. This means that the results of the previous section extend to the context of random dynamical systems. We have seen in the previous section that that the derivative of the extremal mappings at the boundary of the minimal invariant sets can be used to indicate a discontinuous topological bifurcation. In this section, we present an algorithm for estimating the derivative of the extremal maps from a time series generated by iterations of a random dynamical system with bounded noise.

Consider a time series  $S = \{ x_0, x_1, x_2, \dots,x_n \}$ as a realization of iterations of
\begin{equation}\label{eq:genRDS}
x_{i+1} =   h(x_i,\xi_i) \fa i\in \{0,\dots,n-1\}\,,
\end{equation}
where $h : \mathbb{R} \times [0,1] \to \mathbb{R}$, and $\{\xi_i\}_{i\in\N_0}$ is a sequence of i.i.d.~random variables with values in $[0,1]$. Note that these random variables naturally have finite expectation and variance. The map $h$ has an associated set-valued mapping $F$ defined by
\[
F(x) = h(x,[0,1]) \fa x\in \R\,.
\]
We suppose in the following that the extremal maps $f^+$ and $f^-$ of $F$, as defined in \eqref{botmap}, are continuously differentiable.

Let $E=[e_1,e_2]$ be a minimal invariant set for $F$ such that $f^-(e_1)=e_1$ and $f^+(e_2)=e_2$. We aim at providing an approximation of the derivative of $f^-$ at $e_1$, depending on the choice of two intervals of radius $\delta$ and a choice of two intervals in the image of the map $h$ of size $\varepsilon$. Note that there is a analogous algorithm for estimating the derivative of $f^+$ at $e_2$.

\begin{algorithm}\label{alg}
  We approximate $\frac{\de f^-}{\de x} (e_1)$ as follows.
  \begin{enumerate}
    \item[(A1)]
    Consider two subintervals of $E$ of length $\delta>0$, given by
    \begin{equation*}
      I_1 := (m_1 ,m_1+\delta)\qquad \text{and}\qquad I_2:= (m_2,m_2+\delta),
    \end{equation*}
    such that $\delta_0 := m_2 - m_1 - \delta > 0$ (this guarantees that the intervals $I_1$ and $I_2$ are separated), and define $$I:=[m_1,m_2+\delta].$$
    We require that $I$ is close $e_1$.
    \item[(A2)]
    Assume that there exist $w_1, \dots, w_{k_1} \in S\cap I_1$ and $z_1, \dots, z_{k_2} \in S\cap I_2$ such that $w_i = x_{\gamma_i}$ and $z_i = x_{\eta_i}$ for some sequences $\{ \gamma_i \}_{i = 1,\dots, k_1}$ and $\{\eta_i\}_{i = 1,\dots, k_2}$, and each consecutive iteration $w^+_{i} = x_{\gamma_i+1}$ is contained in
    $$
      \tilde{I_1} :=\Big[\min_{x \in I_1} f^-(x)  ,\min_{x \in I_1} f^-(x) + \varepsilon \Big]\,,
    $$
    and $z^+_{i} = x_{\eta_i+1}$ is contained in
    $$
      \tilde{I_2} :=\Big[\min_{x \in I_2} f^-(x) ,\min_{x \in I_2} f^-(x) + \varepsilon \Big]
    $$
    for a given and fixed $\varepsilon > 0$.
    \item[(A3)]
    The approximation of the derivative of the function $f^-$ in $e_1$ is then given by
    \[
      D := \frac{\frac{1}{k_1}  \sum_{i=1}^{k_1} z_i^+  -  \frac{1}{k_2} \sum_{i=1}^{k_2} w_i^+ }{\Delta}\,,
    \]
    where $\Delta = m_2 +\delta - m_1$ is the length of the interval $I$.
  \end{enumerate}
\end{algorithm}
We note that the approximation $D$ depends on the time series $S$, on $\varepsilon$ and on the choices of the intervals $I_1$ and $I_2$.

It is well known that the Markov semi-group generated by \eqref{eq:genRDS} admits an ergodic stationary measure that is supported on the minimal invariant set $E$ (see \cite{ZmarrouHomburg}). We assume that the stationary measure is absolutely continuous with respect to the Lebesgue measure on $E$, and the Lebesgue density does not vanish in the interior of $E$. In addition, \eqref{eq:genRDS} induces a Markov chain $\{ X_i \}_{i \in \N_0}$
that is strictly stationary and ergodic.

We rewrite the map $h$ as
\[
h (x, \xi) = f^-(x) + g(x,\xi)\,,
\]
where $g : \mathbb{R} \times [0,1] \to \mathbb{R}$ is a map of the same regularity as $h$. We introduce the random variable
\begin{equation}\label{eq:xplus}
  X^+ := f^-(X) + g(X, \xi )\,,
\end{equation}
where $X$ is distributed according to the stationary distribution of the Markov semi-group generated by \eqref{eq:genRDS}. Note that $X^+$ has the same distribution as $X$, but we will mainly be interested in the distribution of $(X,X^+)$, and it can be shown that the delayed Markov chain $\{(X_i,X_{i+1})\}_{i\in\N_0}$ is stationary and ergodic with this distribution.

We rewrite the difference quotient introduced in (A3) in terms of the above random variables, and we show that the quantity calculated in (A3) is an estimate from below of the derivative of the extremal map $f^-$ in the interval $I$. Note that we require $I_1$ and $I_2$, and hence $I$, to be close to $e_1$, as stated in (A1), if we want to approximate $\frac{\de f^-}{\de x} (e_1)$.

In (A3), two time averages are performed in order to approximate the derivative. We now use the ergodic theorem to justify this approach.

\begin{proposition}\label{prop:algo}
  Consider the ergodic stationary process $\{ X_i \}_{i \in \N_0}$, and let $I_1$, $I_2$, $\tilde{I_1}$ and $\tilde{I_2}$ be the intervals introduced in (A1) and (A2). Define for each $i\in\N$ the random variables
  \[
  \chi_{1}(i) := \1_{I_1} (X_i) \1_{\tilde{I_1}} ( X_{i+1}) \quad \mbox{ and } \quad
  \chi_{2}(i) := \1_{I_2} (X_i) \1_{\tilde{I_2}} ( X_{i+1})\,,
  \]
  where $\1_A$ denotes the indicator function, i.e.~$\1_A(x) = 1$ if $x\in A$, and $\1_A(x) = 0$ if $x\notin A$.
  Define for each $n\in \N$ the random variables
  \begin{equation}\label{eq:kn}
    k_1(n) :=  \sum_{i = 1}^n \chi_{1}(i) \quad \mbox{ and }\quad k_2(n)  :=  \sum_{i = 1}^n \chi_{2}(i)\,,
  \end{equation}
  and consider the random variables $X$ and $X^+$ as introduced in \eqref{eq:xplus}. Then the following limits exist almost surely, and the two inequalities
  \begin{equation}
  \label{eq:sum1}
    \lim_{n \to \infty} \frac{1}{k_1(n)} \sum_{i=1}^n \chi_{1}(i) h(X_i, \xi_i ) \ge \frac{ \mathbb{E}\big[ \1_{I_1} (X) \1_{\tilde{I_1}} ( X^+) f^-(X) \big] }{\mathbb{P}\big(X \in I_1, X^+ \in \tilde{I_1} \big)}
  \end{equation}
  and
  \begin{equation}
  \label{eq:sum2}
  \lim_{n \to \infty} \frac{1}{k_2(n) } \sum_{i=1}^n \chi_{2}(i)  h(X_i,\xi_i)   \le \frac{\mathbb{E}\big[ \1_{I_2} (X) \1_{\tilde{I_2}} ( X^+) f^-(X) \big]}{\mathbb{P}\big(X \in I_2, X^+ \in \tilde{I_2} \big)}  + \varepsilon
  \end{equation}
  hold.
\end{proposition}

\begin{remark}
  Denote the joint Lebesgue density of $X$ and $X^+$ by $\phi : \mathbb{R}\times \mathbb{R} \to \mathbb{R}$. We aim at rewriting the quantity
  \[
    \frac{ \mathbb{E} \left[ \1_{I_1} (X) \1_{\tilde{I_1}} ( X^+) f^-(X) \right] }{\mathbb{P}\big(X \in I_1, X^+ \in \tilde{I_1} \big)}\,.
  \]
  Firstly, the numerator reads as
  \[
    \mathbb{E} \left[ \1_{I_1} (X) \1_{\tilde{I_1}} ( X^+) f^-(X) \right] =  \int_{\mathbb{R}} \int_{\mathbb{R}} \1_{I_1} ( x) \1_{\tilde{I_1}} ( y)   f^-(x) \phi(x,y) \,\textrm{d} x \,\textrm{d} y\,.
  \]
  Note that the function $x \mapsto \int_{\tilde{I_1}}  \phi(x,y) \textrm{d} y$ is strictly positive whenever $x \in I_1$, since $\mathbb{P}\big( X_{i+1} \in \tilde{I_1} | X_i \in I_1\big) > 0$.
  The extra term $\1_{I_1} ( X) \1_{\tilde{I_1}} ( X^+) $ inside the integral restricts the domain of integration, and hence, $\phi$ is not a probability density on $I_1 \times \tilde{I_1}$. However, $\mathbb{P}(X_1 \in I_1, X_2 \in \tilde{I_1} ) $ is the desired normalization factor, as we have
  \[
  \mathbb{P}(X_1 \in I_1, X_2 \in \tilde{I_1} ) = \int_{I_1} \textrm{d}x  \int_{\tilde{I_2}}  \phi(x,y)\, \textrm{d}y > 0\,.
  \]
  Hence, the function
  \[
  x \mapsto \frac{1}{\int_{I_1} \textrm{d}x  \int_{\tilde{I_2}}  \phi(x,y) \textrm{d}y } \int_{\tilde{I_2}}  \phi(x,y) \,\textrm{d} y
  \]
  is  a probability density describing the statistics of those points in $I_1$ for which the consecutive iteration is in $\tilde{I_1}$.
\end{remark}

\begin{proof}
  We prove \eqref{eq:sum2}. Note that
  \begin{displaymath}
    \lim_{n \to \infty} \frac{1}{k_2(n)} \sum_{i=1}^n \chi_{2}(i) h(X_i,\xi_i )
    =\lim_{n \to \infty} \frac{1}{k_2(n)} \frac{n}{n} \sum_{i=1}^n \chi_{2}(i) \big( f^-(X_i) + g(X_i,\xi_i )\big)\,.
  \end{displaymath}
  We consider the limits
  \begin{equation}
  \lim_{n \to \infty} \frac{1}{n} k_2(n) \quad\mbox{ and }\quad
  \lim_{n \to \infty} \frac{1}{n} \sum_{i=1}^n \chi_{2}(i) ( f^-(X_i) + g(X_i,\xi_i ))
  \end{equation}
  separately.
  For both limits the ergodic theorem is applied to the delayed Markov chain $\{(X_i,X_{i+1})\}_{n\in\N_0}$, and we consider the observables
  \begin{align*}
    O_1 ( X_i, X_{i+1} )  &=  \1_{I_2} (X_i) \1_{\tilde{I_2}} (X_{i+1})\,,\\
    O_2( X_i , X_{i+1}) &= \1_{I_2} ( X_i ) \1_{\tilde{I_2}} ( X_{i+1}) ( f^-( X_i ) + g ( X_i , \xi_i )  ) = \1_{I_2} ( X_i ) \1_{\tilde{I_2}} ( X_{i+1}) X_{i+1}\,.
  \end{align*}
  The first observable leads to
  \[
  \lim_{n \to \infty} \frac{1}{n} k_2(n) =   \lim_{n \to \infty} \frac{1}{n} \sum_{i = 1}^n \1_{I_2} ( X_i) \1_{\tilde{I_2}} (X_{i+1})  =  \mathbb{P}( X \in I_2 , X^+ \in \tilde{I_2} )\,.
  \]
  The second observable yields
  \begin{align*}
  &\lim_{n \to \infty} \frac{1}{n}  \sum_{i=1}^n \chi_{2}(i) ( f^-(X_i) + g(X_i,\xi_i ))  \le  \lim_{n \to \infty}  \frac{1}{n}  \sum_{i=1}^n \chi_{2}(i) ( f^-(X_i) + \varepsilon )   \\
    &=  \mathbb{E}[  \1_{I_2} (X) \1_{\tilde{I_2}} ( X^+) f^-(X) ]  + \varepsilon \mathbb{P} (X \in I_2 , X^+ \in \tilde{I_2} )\,.
  \end{align*}
  combining these two limit relations gives
  \begin{displaymath}
   \lim_{n \to \infty} \frac{1}{k_2(n)} \frac{n}{n} \sum_{i=1}^n \chi_{2}(i) h(X_i,\xi_i)
    \le  \frac{\mathbb{E}[  \1_{I_2} (X) \1_{\tilde{I_2}} ( X^+) f^-(X) ]}{\mathbb{P}(X \in I_2, X^+ \in \tilde{I_2} )}  + \varepsilon\,.
  \end{displaymath}
  The proof for \eqref{eq:sum1} is analogous, with the difference that the perturbation $g(x, \xi )$ is estimated from below by zero.
\end{proof}

The following corollary is a direct consequence of Proposition \ref{prop:algo} and it shows that  $D$ in (A3) is an estimate from below of the derivative of $f^-$ in the interval $I$.

\begin{corollary}\label{coro:algo}
  Under the conditions of Proposition~\ref{prop:algo}, almost surely, we have
  \begin{align*}
  &\lim_{n \to \infty} \frac{1}{\Delta} \left( \frac{1}{k_2(n)} \sum_{i=1}^n \chi_{2}(i) h(X_i,\xi_i)
    - \frac{1}{k_1(n)} \sum_{i=1}^n \chi_{1}(i) h(X_i,\xi_i)   \right)\\
   &\le  \max_{x \in I} (f^-)^\prime ( x )  + \frac{\varepsilon}{\Delta}\,.
  \end{align*}
\end{corollary}
\begin{proof}
Proposition \ref{prop:algo} yields
\begin{align*}
&\lim_{n \to \infty}  \frac{1}{\Delta} \left( \frac{1}{k_2(n)} \sum_{i=1}^n \chi_{2}(i) h(X_i,\xi_i )  - \frac{1}{k_1(n)} \sum_{i=1}^n \chi_{1}(i) h(X_i,\xi_i )  \right) \\
& \le \frac{1}{\Delta} \left( \frac{\mathbb{E}[  \1_{I_2} (X) \1_{\tilde{I_2}}(X^+)  f^-(X) ]}{\mathbb{P}(X \in I_2, X^+ \in \tilde{I_2} )}  -  \frac{\mathbb{E}[  \1_{I_1} (X) \1_{\tilde{I_1}}(X^+)  f^-(X) ]}{\mathbb{P}(X \in I_1, X^+ \in \tilde{I_1} )} \right)  + \frac{\varepsilon}{\Delta} \numberthis \label{expect1} \\
& \le  \max_{x \in I_1, y \in I_2 }   \frac{1 } { \Delta } \big( f ( y) - f (x ) \big)     + \frac{\varepsilon}{\Delta}
 \le \max_{x \in I_1, y \in I_2 }  \frac{ f (y) - f (x )}{y-x}   + \frac{\varepsilon}{\Delta}\\
& \le \max_{x \in I } (f^-)^\prime ( x )  + \frac{\varepsilon}{\Delta}\,.
\end{align*}
We approximated the expectation in \eqref{expect1} by considering the maximum. Next we used $\Delta  > y -x$ and the intermediate value theorem.
\end{proof}

\begin{remark}
  The error $\frac{\varepsilon}{\Delta}$ depends on the size of the intervals $I_1, I_2,\tilde{I}_1$ and $\tilde{I}_2$. In order to reduce this error, one could shrink the size of the intervals, but as a consequence, we will need to wait a longer time to obtain a sample of points in satisfying the conditions in (A2).
\end{remark}

\begin{remark}
  In Corollary~\ref{coro:algo}, the estimate of the derivative takes two types of errors into account, one coming from the noise, and the other given by the underestimation of the difference quotient $D$ (note that $\Delta$ serves as upper bound in the denominator). We want to compare the derivative of the extremal map $f^- ( x) = h(x,0)$ with the difference quotient introduced in (A3) for $k_1=k_2=1$. For any $x \in I_1$,  $y \in I_2$ and $\xi \in [0,1]$, we have
  \begin{align*}
  & \left| \max_{c \in [m_1, m_2+\delta] } \frac{\partial h}{ \partial x} (c, 0 ) -  \frac{h (y, \xi ) - h (x, \xi ) }{\Delta} \right| =  \\
  = & \left| \max_{c \in [m_1, m_2+\delta] } \frac{\partial h}{ \partial x} (c, 0 ) - \frac{y - x}{y-x} \frac{h (y, \xi ) - h (x, \xi ) }{\Delta} \right| \\
  = & \left| \max_{c \in [m_1, m_2+\delta] } \frac{\partial h}{\partial x} ( c, 0 ) - \frac{y - x}{\Delta}  \frac{\partial h}{\partial x} ( \bar{c}, \xi ) \right| =   \\
  = & \left| \max_{c \in [m_1, m_2+\delta] } \frac{\partial h}{\partial x} ( c, 0 ) \right| \left| 1  - \frac{y - x}{\Delta}  \frac{\frac{\partial h}{\partial x} ( \bar{c}, \xi )}{\max_{c \in [m_1, m_2+\delta] } \frac{\partial h}{\partial x} ( c, \xi )} \right|\,.
  \end{align*}
  Here, we used the intermediate value theorem to obtain $\bar{c} \in [m_1, m_2+\delta]$. Note that if the map $h$ is at least continuously differentiable, the term on the left hand side is bounded, and then the error can be controlled by the quantity
  \[
  \left| 1  - \frac{y - x}{\Delta}  \frac{\frac{\partial h}{\partial x} ( \bar{c}, \xi )}{\max_{c \in [m_1, m_2+\delta] } \frac{\partial h}{\partial x} ( c, \xi )} \right|.
  \]
  The error depends on the size of the intervals, since $\frac{y - x}{\Delta}$ appears and on the variation of the function $h$ as the ratio of the derivatives appears. We want to emphasize that the shorter the intervals $I_1, I_2, \tilde{I}_1,\tilde{I}_2$ are, the smaller is the error is. On the other hand, from the application point of view, the probability of having consecutive points in $I_i$ and $\tilde{I}_1$ is also lower, since  $\mathbb{P}(X \in I_i, X^+ \in \tilde{I_i})$ for $i\in\{1,2\}$, tends to zero as the size of the intervals tend to zero. When applying Algorithm~\ref{alg} to a time series, it is then crucial to consider the size of the intervals in order to have a suitable sample, and the variation of the extremal function in the intervals, in order to have a meaningful estimate.
\end{remark}

\begin{remark}
  In the case of additive noise, Algorithm~\ref{alg} can be simplified significantly. In (A2), the intervals $\tilde{I_1}$ and $\tilde{I_2}$, and the condition that the next iteration must be contained in $\tilde{I_1}$ and $\tilde{I_2}$ can be dropped. In order to show this we state and sketch the proof of Proposition \ref{prop:algo} and Corollary \ref{coro:algo} in the case of additive noise.
  Consider the stationary process $\{X_i\}_{i\in\N_0}$ as previously defined, and let $\xi_i$ be a sequence of i.i.d.~bounded random variables distributed according to a random variable $\xi$ such that $X_i$ is independent of $\xi_i$ for every $i \in \mathbb{N}$. In the additive noise case, the map $h$ reads as
  \[
  h(X_i, \xi_i) =f(X_i) + \xi_i\,.
  \]
  Define for $i \in \mathbb{N}$ the random variables
  \[
  \Gamma_1(i) := \1_{I_1}(X_i) \quad \mbox{and}\quad \Gamma_2(i) := \1_{I_2} (X_i)\,.
  \]
  Then the statements \eqref{eq:sum1} and \eqref{eq:sum2} of Proposition \ref{prop:algo} read as
  \begin{equation}
  \label{eq:sumadd1}
  \lim_{n \to \infty} \frac{1}{k_1(n)} \sum_{i = 1}^n \Gamma_1 (i) (f(X_i) + \xi_i)   = \frac{\mathbb{E} [\1_{I_1}(X) f(X)  ]}{\mathbb{P}(X \in I_1)} + \mathbb{E}[\xi]
  \end{equation}
  and
  \begin{equation}
  \label{eq:sumadd2}
  \lim_{n \to \infty} \frac{1}{k_2(n)} \sum_{i = 1}^n \Gamma_2(i) (f(X_i) + \xi_i)   = \frac{\mathbb{E} [\1_{I_2}(X) f(X)  ]}{\mathbb{P}(X \in I_2)}+ \mathbb{E}[\xi]\,,
  \end{equation}
  where $k_1,k_2$ are defined in \eqref{eq:kn}. The calculation of the limits \eqref{eq:sumadd1} and \eqref{eq:sumadd2} is analogous to the calculation performed in the proof of Proposition~\ref{prop:algo}. When compared to \eqref{eq:sum1} and \eqref{eq:sum2}, note that in the right hand side of the equalities \eqref{eq:sumadd1} and  \eqref{eq:sumadd2}, a different contribution of the noise appears. For  \eqref{eq:sumadd1}, we have
  \[
  \lim_{n \to \infty} \frac{1}{k_1(n)} \sum_{i = 1}^n \Gamma_1 (i)  \xi_i  = \frac{1}{\mathbb{P} (X \in I_1)} \mathbb{E}[\xi] \mathbb{P}(X \in I_1) = \mathbb{E}[\xi]\,,
  \]
  and similarly for \eqref{eq:sumadd2}, we have
  \[
  \lim_{n \to \infty} \frac{1}{k_2(n)} \sum_{i = 1}^n \Gamma_2 (i)  \xi_i  = \mathbb{E}[\xi]\,.
  \]
  We do not need an estimate of the limits, since the effect of the noise will cancel out. On the other hand, the additive version of Corollary \ref{coro:algo} reads as
   \begin{align*}
    &\lim_{n \to \infty} \frac{1}{\Delta} \left( \frac{1}{k_2(n)} \sum_{i=1}^n \Gamma_{2}(i) (f(X_i) + \xi_i)
      - \frac{1}{k_1(n)} \sum_{i=1}^n \Gamma_{1}(i) (f(X_i) + \xi_i)  \right)\\
     &\le  \max_{x \in I} (f^-)^\prime ( x )\,.
    \end{align*}
  In fact, one has
  \begin{align*}
  &\lim_{n \to \infty}  \frac{1}{\Delta} \left( \frac{1}{k_2(n)} \sum_{i=1}^n \Gamma_{2}(i) (f(X_i) + \xi_i)   - \frac{1}{k_1(n)} \sum_{i=1}^n \Gamma_{1}(i) (f(X_i) +  \xi_i)   \right) \\
  & \le \frac{1}{\Delta} \left[ \frac{\mathbb{E} [\1_{I_2}(X) h(X)   ]}{\mathbb{P}(X \in I_2)} + \mathbb{E}[\xi]  - \frac{\mathbb{E} [\1_{I_1}(X) h(X)  ]}{\mathbb{P}(X \in I_1)} -  \mathbb{E}[\xi]  \right]   \numberthis \label{expectadd} \\
  & \le  \max_{x \in I_1, y \in I_2 }   \frac{1 } { \Delta } \big( f ( y) - f (x ) \big)     \le \max_{x \in I_1, y \in I_2 }  \frac{ f (y) - f (x )}{y-x}    \\
  & \le \max_{x \in I } (f^-)^\prime ( x )\,.
  \end{align*}
  The proof is the same as in Corollary \ref{coro:algo}, except for step \eqref{expectadd}, where the contribution of the noise cancels out. Hence we have that the intervals  $\tilde{I_1}$ and $\tilde{I_2}$ and the condition on the next iterated in (A2) can be dropped.

  The absence of the intervals  $\tilde{I_1}$ and $\tilde{I_2}$ implies that the algorithm is more effective in the additive case for two reasons. Firstly, there are less conditions on the points of the time series, implying that one has generally a greater number of points eligible for the calculating the estimate when compared to the general case. Secondly, we obtain a better estimate from below of the derivative, since the term $\frac{\varepsilon}{\Delta}$ does not appear in the additive case.
\end{remark}

%%%%%%%%%%%%%%%%%%%%%%%%%%%%%%%%%%%%%%%%%%%%%%%%%%%%%%%%%%%%%%%%%%%%%%%%%%%%%%%%%%%%%%%%%%%%
\section{Early-warning signals for bifurcations in the Koper model}
\label{sec:koper}

In this section, we numerically study the stochastic return map of the Koper model
with additive bounded noise. The Koper model is a prototypical example of a system
exhibiting mixed-mode oscillations (MMOs), i.e.~periodic orbits with widely varying amplitudes in one period.
Depending upon the value of a parameter $\lambda \in \mathbb{R}$
the system exhibits different stable modes of oscillations. Bifurcations of the
minimal invariant set of the set-valued mapping associated to the stochastic
return map correspond to the system shifting from an oscillation mode to another.
With the algorithm of Section~\ref{sec:algorithm} we reconstruct the derivative
of the extremal map $f^-_\lambda$ for different values of $\lambda$ to find an
early-warning sign for upcoming bifurcations.

We consider the Koper model given by the ordinary differential equation
\begin{align*}
\varepsilon \frac{\txtd x}{\txtd \tau}&  = y-x^3+3x=:f(x,y,z), \\
\frac{\txtd y}{\txtd \tau}& = kx - 2 (y + \lambda) + z, \label{KoperODE} \numberthis \\
\frac{\txtd z}{\txtd \tau}& = \lambda + y - z,
\end{align*}
where $k,\lambda \in \mathbb{R}$ are real parameters and $\varepsilon$ is assumed to
be small and positive. We fix $k = -10$
and study the system for a special parameter range $\lambda \in I =  (-8, -23/6 )$.
This corresponds to a typical transition through different MMOs as discussed in
\cite{KuehnRetMaps,Desrochesetal}. Observe that \eqref{KoperODE} is a fast-slow
system due to the small parameter $\varepsilon$ with one fast variable $x$, and two
slow variables $y,z$. Therefore, the analysis of MMOs can be carried out using
the framework of multiple time scale dynamical systems \cite{Desrochesetal,KuehnBook}.
We provide a brief description of the multiscale geometry of the system. The critical
manifold is given by
\[
C_0 := \{ (x,y,z) \in \mathbb{R}^3 | y = x^3 - 3x  \}.
\]
Note that $C_0$ can be viewed as the algebraic constraint for the slow subsystem
obtained from \eqref{KoperODE} in the limit $\varepsilon\ra 0$. Furthermore, we can
also re-scale time $t:=\tau/\varepsilon$ and on the time scale $t$, the critical manifold
consists of equilibrium point for the fast subsystem
\begin{align*}
\frac{\txtd x}{\txtd t}&  = y-x^3+3x, \\
\frac{\txtd y}{\txtd t}& = 0, \label{KoperODEfs} \numberthis \\
\frac{\txtd z}{\txtd t}& = 0,
\end{align*}
obtained again by a formal singular limit $\varepsilon \ra 0$. $C_0$ is normally
hyperbolic (which just means here that $(\partial_x f)|_{C_0}$ is nonzero) away from the
two fold lines
\[
F^+ := \{x = 1  \}\cap C_0 \qquad \text{and} \qquad F^- := \{x =- 1  \}\cap C_0.
\]
The folds $F^\pm$ separate the normally hyperbolic repelling manifold
$C^r_0 = C_0 \cap \{ x<1 \} \cap \{  x>-1 \}$ from the two normally hyperbolic
attracting manifolds $C^{a+}_0 = C_0 \cap \{ x > 1\} $ $C^{a-}_0 = C_0 \cap \{ x <-1\} $.
In particular, $C_0^r$ consists of repelling equilibria of \eqref{KoperODEfs} while
$C_0^{a\pm}$ consist of attracting equilibria. There are two folded singularities
for the system located at
\[
q_\pm = (\pm 1 , \pm 2 , 2 \lambda \pm  6),
\]
which can be viewed as equilibria of a suitable desingularized slow subsystem; see
{e.g.}~\cite{Desrochesetal,KuehnBook,Wechselberger2} for more details. The main point
for us here is that it is well-understood how the interplay between folded singularities
and a global fast-slow return mechanism can generate MMOs as shown in Figure \ref{figure:modeofoscillation}. Due to
the system symmetry
\[
(x,y,z,\lambda, k ) \to (-x , -y, -z , - \lambda , k),
\]
we can consider $q_+$ only. It can be shown that for $\lambda \in I =  (-8, -23/6 )$
the folded singularity $q_+$ is a so-called \emph{folded node} \cite{KuehnRetMaps}.
It is known \cite{BronsKrupaWechselberger,SzmolyanWechselberger1} that in the vicinity
of a folded node various small-amplitude oscillations (SAOs) can be generated. The idea
is that a set of special canard orbits partitions the neighborhood of $q_+$ in the
$(x,z)$ plane into certain sectors of rotation. An orbit of system \eqref{KoperODE}
entering a sector will perform $s$ small oscillations before making a transition to a fast
segment. To each sector corresponds a different number $s$ of oscillations.

Due to the geometry of the $S$-shaped critical manifold and its attractivity properties
it follows that we also expect large-amplitude oscillations (LAOs) in the system due to
a relaxation-oscillation type switching between $\cO(\varepsilon)$-neighbourhoods of the
two attracting branches $C^{a\pm}_0$ (see also Figure~\ref{figure:modeofoscillation}).
The combination of the SAO and the LAO mechanisms can then yield MMOs.

It is possible to define a return map $P_\lambda$ on a suitable codimension-one section
$\Sigma$ to study the MMOs via the return map
\[
P_\lambda : \Sigma \to \Sigma.
\]
Here it will be convenient to fix $\Sigma$ as
\[
\Sigma = \left\{ (x,y,z) \in \mathbb{R}^3 | x = - \frac{4}{5} , y \in  (-3 , -1), z \in (z_0,z_1) \right\}.
\]
The points $z_0,z_1$ will be chosen according to $\lambda$ and the noise amplitude. Due to
the fast-slow structure, the map $P_\lambda$ is almost one-dimensional \cite{Guckenheimer8}
and it is common to project the map $P_\lambda$ further to reduce it to a one-dimensional map
\[
p_\lambda : \Sigma \cap \{ y = y_0 \}  \to \mathbb{R},
\]
where we define the map as
\[
z \mapsto p_\lambda (z) := (\pi_z \circ P_\lambda)(y_0, z)
\]
with $\pi_z$ denoting the projection onto the $z$ axis. The projection is made in order to
simplify the analysis even if the resulting $p_\lambda$ is non-invertible. Sometimes it is
even possible to calculate an explicit (asymptotic) expression for the nearly one-dimensional
map in certain parameter regimes as explained for a slightly different, but related, three-time
scale model in \cite{KrupaPopovicKopell}.

\begin{figure}[htbp]
\begin{minipage}{0.5\textwidth}
   \begin{overpic}[width=\textwidth,unit=1mm]{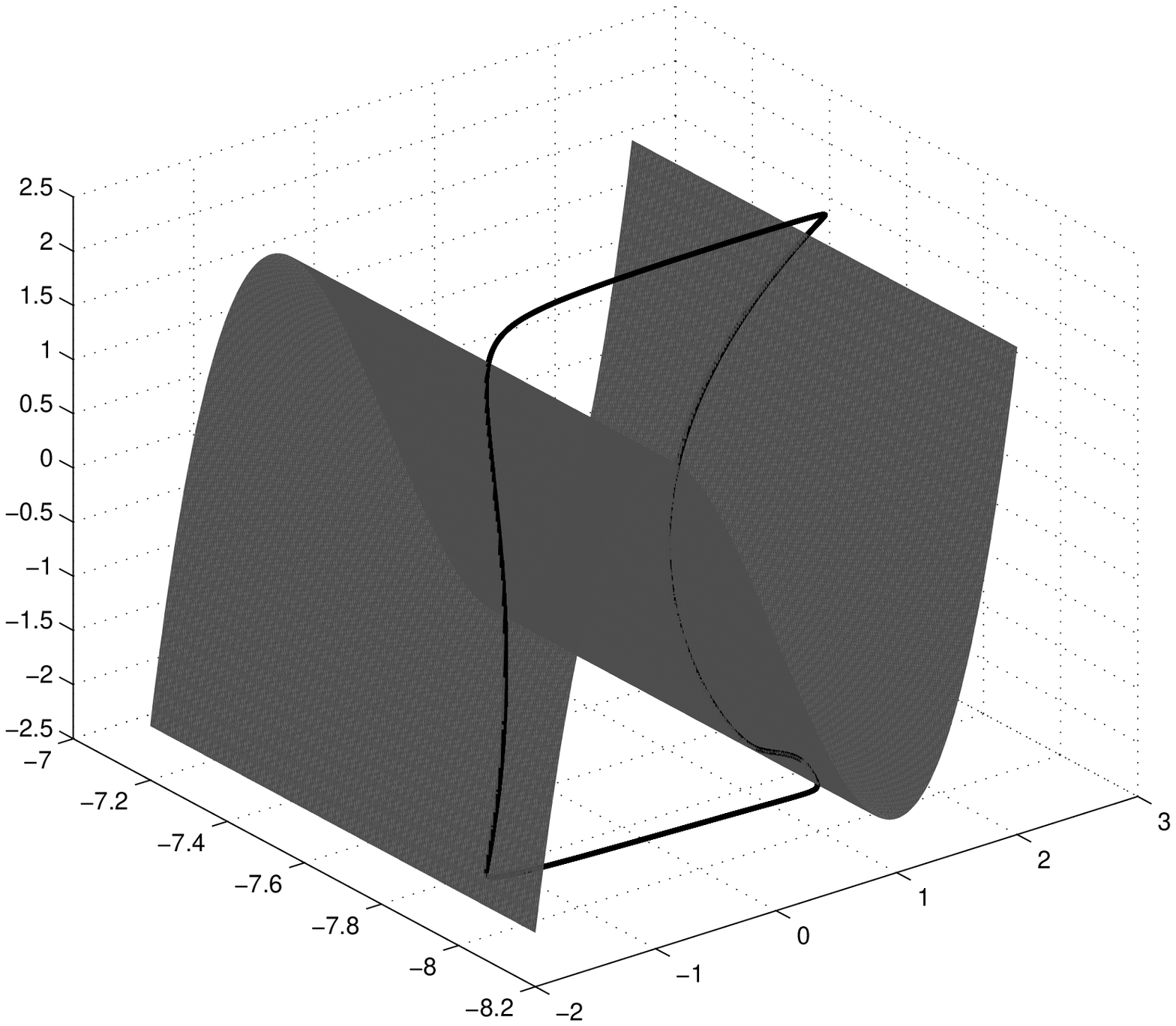}
\put(28,4){{\parbox{0.4\linewidth}{%
     \[
       x
     \]}}}
\put(5,5){{\parbox{0.4\linewidth}{%
     \[
      z
     \]}}}
\put(-10,29){{\parbox{0.4\linewidth}{%
     \[
        y
     \]}}}
\end{overpic}
\end{minipage}
\begin{minipage}{0.5\textwidth}
   \begin{overpic}[width=\textwidth,unit=1mm]{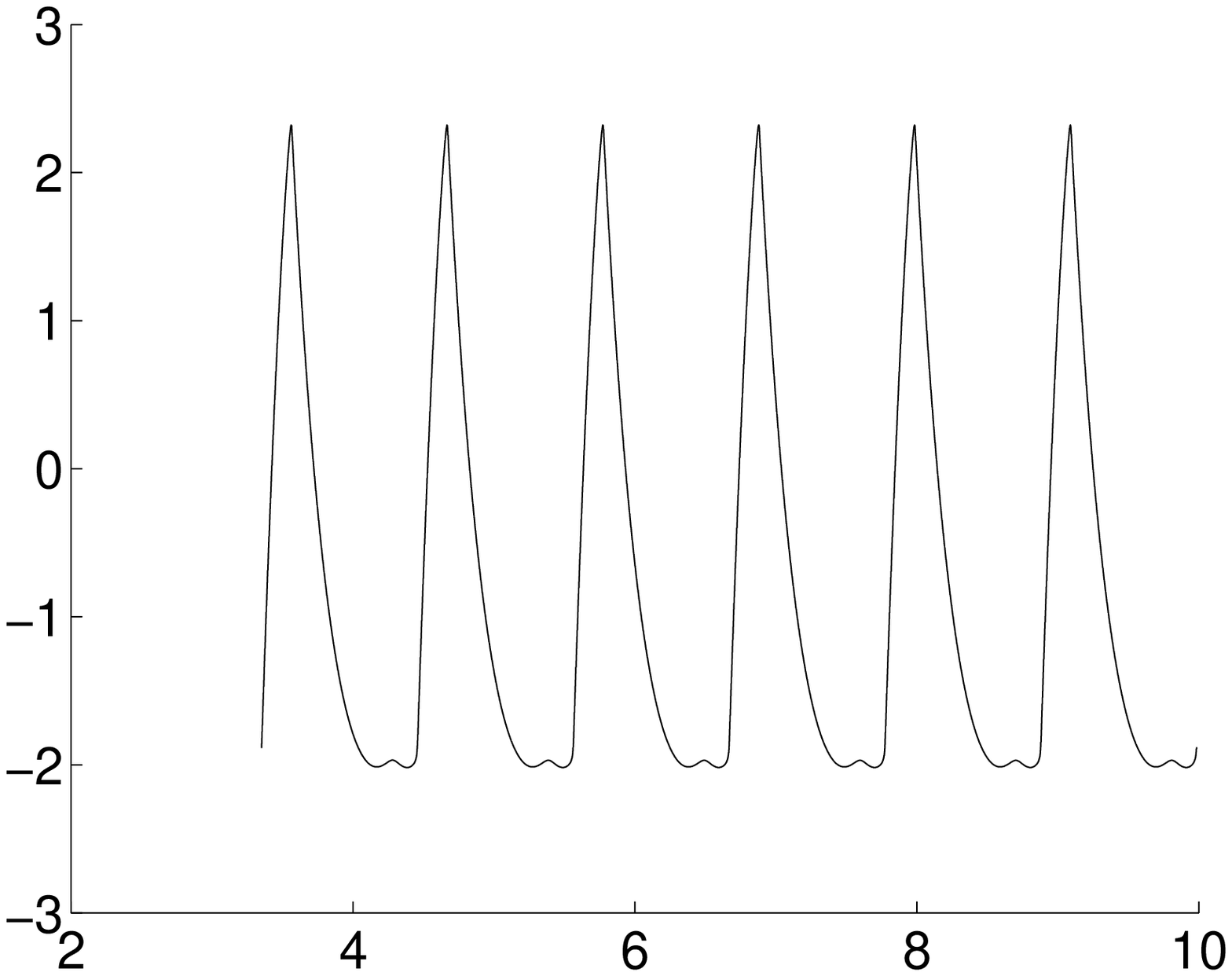}
\put(20,0){{\parbox{0.4\linewidth}{%
     \[
       t
     \]}}}
\put(-8,23){{\parbox{0.4\linewidth}{%
     \[
        y
     \]}}}
\end{overpic}
\end{minipage}
\caption{\label{figure:modeofoscillation}A trajectory of a solution of
system \eqref{KoperODE} and its projection on the $y$ component, corresponding to the initial condition $(x,y,z ) =
( -0.81,-2,-8)$ and $\lambda = -6.9$.}
\end{figure}
\begin{figure}[htbp]
\centering
   \begin{overpic}[width=0.75\textwidth,unit=1mm]{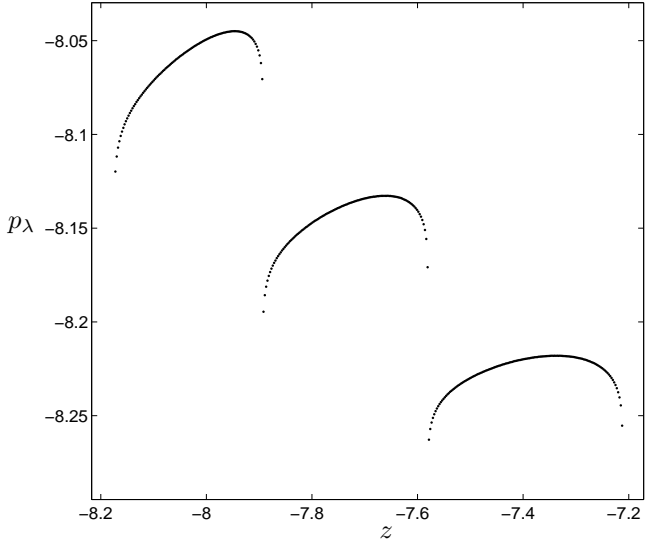}
\put(26,4){{\parbox{0.4\linewidth}{%
     \[
       z
     \]}}}
\put(-22,45){{\parbox{0.4\linewidth}{%
     \[
        p_\lambda
     \]}}}
\end{overpic}
\caption{\label{deterministickoper} Deterministic return map for the Koper system $\lambda = -6.9$. The jumps correspond to the location of the
canard orbits. In fact, the map is continuous and the gaps represent steep
jumps \cite{Guckenheimer8,Desrochesetal}.}
\end{figure}

The geometry of the sectors of rotation for the SAOs is reconstructed very well by the
graph of the return map. Each sector is bounded by two consecutive apparent discontinuities in the
graph of the mapping $ z \mapsto p_\lambda(z)$. The jumps correspond to the location of
canard orbits; see Figure \ref{deterministickoper}.
An attracting fixed point of the return map $p_\lambda$ corresponds to a stable mode
of oscillations. It is key that by varying $\lambda$ the fixed point of the return map
can change and the system can transition to a different stable mode of oscillation. The
goal is to introduce (bounded) noise in the system of ordinary differential equations to predict the shift to
a new mode of oscillations when $\lambda$ is slowly increased towards a bifurcation.
Consider again the system \eqref{KoperODE}.

We introduce bounded additive noise in the
model and consider the family of random maps
\begin{equation}
\label{koperRDE}
\left( \begin{matrix}  x(t) \\ y(t) \\ z(t)  \end{matrix} \right)
=  \left( \begin{matrix}  x(s) \\ y(s) \\ z(s)  \end{matrix} \right)+\left( \begin{matrix} y(s)-x(s)^3+3x(s)  \\  \varepsilon\big(kx(s) - 2 (y(s) + \lambda) + z(s)\big) \quad
 \\  \varepsilon\big(\lambda + y(s) - z(s)\big) \end{matrix} \right) (t-s)+ A (\xi(t)-\xi(s))\,,
\end{equation}
where $\xi = (\xi_1(t),\xi_2(t),\xi_3(t))^\top \in [-1,1]^3$ is a three-dimensional
bounded stochastic process with independent increments, $t>s$, and we fix the time step $(t-s)$.
Each increment $\xi(t)-\xi(s)$ is uniformly distributed in $[-1,1]$. We consider the
case where
\[
A = \sigma \left( \begin{matrix} 1.0 & 0.5 & 0.2 \\ 0.5 & 1.0 & 0.3
\\ 0.2 & 0.3 & 1.0  \end{matrix} \right).
\]
Note that in the limit $(t-s)\ra 0$ we obtain unbounded noise by the central limit
theorem so we cannot directly generalize to bounded-noise random ordinary differential
equations but stay on the level of maps to guarantee bounded noise.

The return map for \eqref{koperRDE} is a \emph{stochastic return map},
$p_\lambda^{st}$ defined on the same section $\Sigma$ of the map $p_\lambda$. The
definition of a stochastic return map is more problematic, since the noise
can destroy the return mechanism or we could have an early return to the section, i.e.~before the trajectory goes through the global return mechanism. A formal construction
is presented in \cite{WeissKnobloch}. In the numerical reconstruction performed in
this paper $\Sigma$ and the initial conditions are chosen in such a way that a realization of the orbit of
the random differential equation does make a global return to $\Sigma$. In practice we consider initial conditions $ (x,y,z)$ with $x< -\frac{3}{4}$, $y< 0$, $z \in (z_0, z_1)$ and we stop the iteration once the orbit $(x(t),y(t),z(t) ) $ is such that $y(t) < 0$ and $x(t)$ crosses $x =   -\frac{3}{4}$.

\begin{figure}[htbp]
  \begin{minipage}{0.5\textwidth}
 \begin{overpic}[width=\textwidth,unit=1mm]{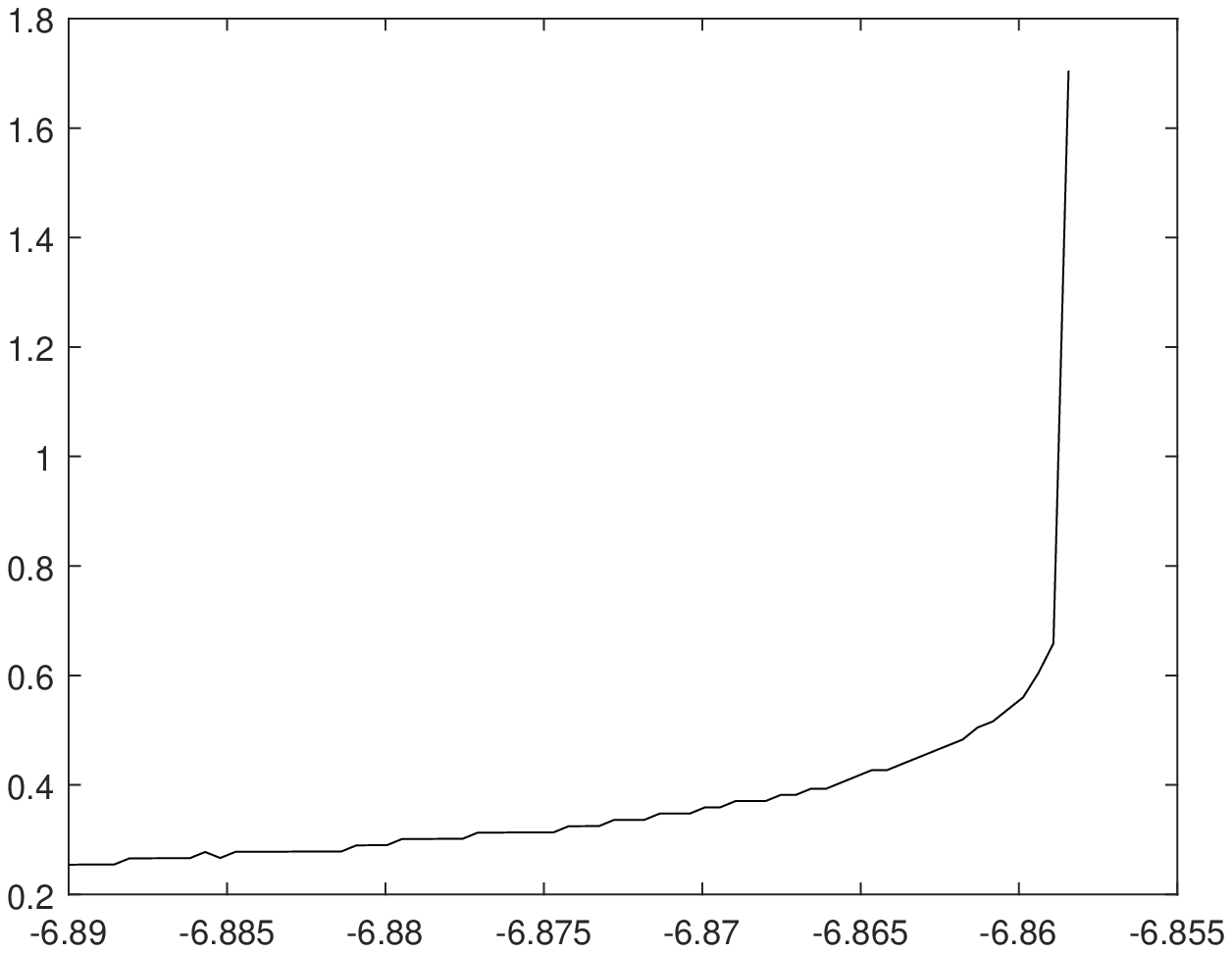}
\put(20,0){{\parbox{0.4\linewidth}{%
     \[
       \lambda
     \]}}}
\put(-11,23){{\parbox{0.4\linewidth}{%
     \[
        D p_{\lambda}
     \]}}}
\end{overpic}
  \end{minipage}
  \begin{minipage}{0.5\textwidth}

 \begin{overpic}[width=\textwidth,unit=1mm]{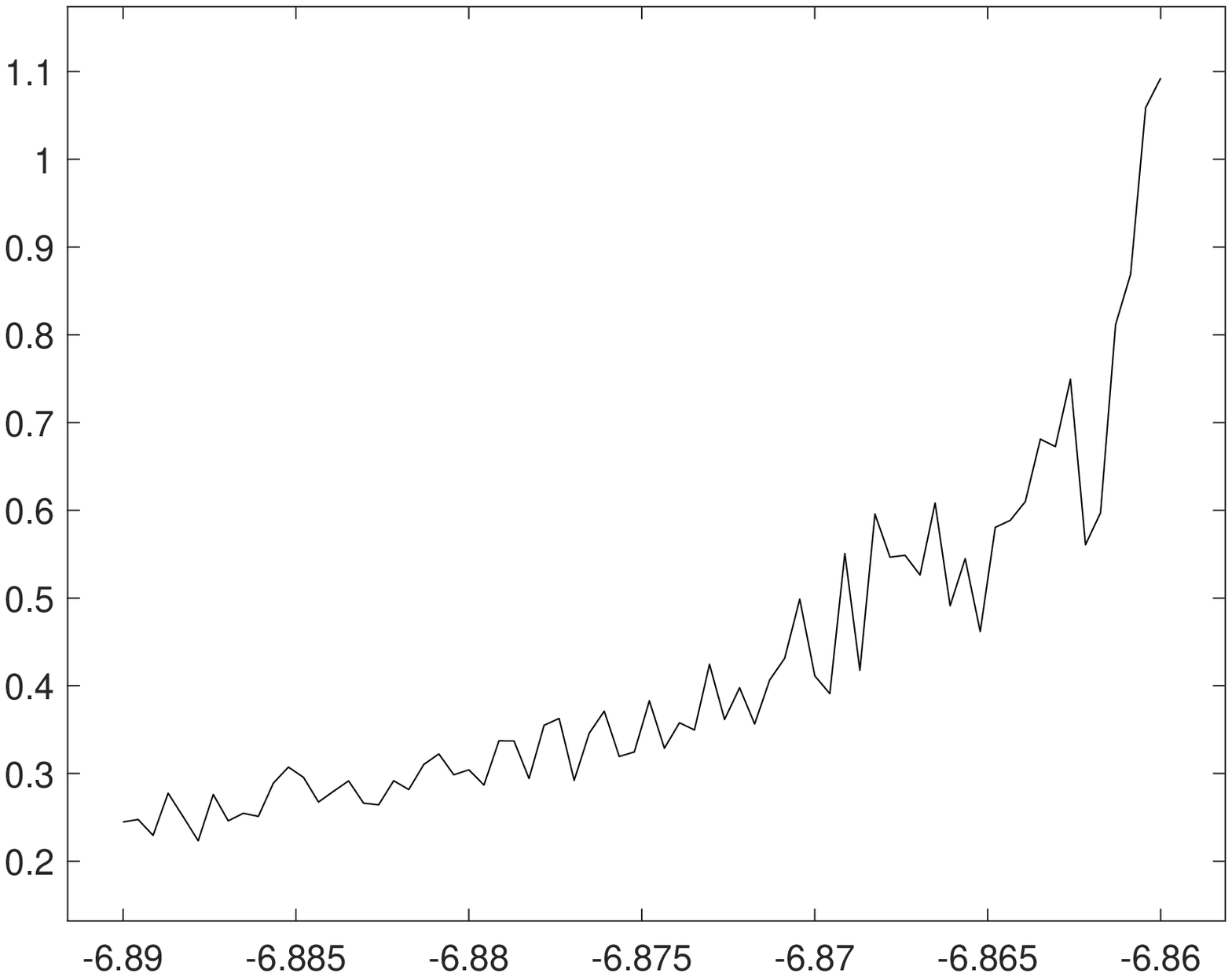}
\put(20,0){{\parbox{0.4\linewidth}{%
     \[
       \lambda
     \]}}}
\put(-11,23){{\parbox{0.4\linewidth}{%
     \[
       D f^-
     \]}}}
\end{overpic}
\end{minipage}
\caption{ \label{figure:koperderivative}(a) Derivative of the deterministic return map $p_\lambda$ evaluated at the fixed point $x_\lambda$ for $\lambda \le \lambda_0 = -6.86$. For $\lambda > -6.86$ the derivative is evaluated in $x_{\lambda_0}$ as the fixed point disappears and the map has a periodic orbit. (b) Derivative of $f^-$ at the lower boundary
of the minimal invariant set. For any given $\lambda$ the lower bound of the minimal
invariant set is retrieved from Figure \ref{figure:minimalinvariant}. We considered
$n=3000$ realizations of the stochastic return map, say $ z_i =  p_\lambda^{st}
(\xi_i,m_\lambda)$, $w_i=p_\lambda^{st} (\xi_i m_\lambda + \varepsilon)$, for
$i = 1,\dots,n$ and approximated the derivative by
$d_\lambda=\frac{\min_i \{ w_i \} -  \min_i \{ z_i \} }{\varepsilon}$. We considered
$\varepsilon = 0.01$.}
\end{figure}

In Figure~\ref{figure:koperlambdastochastic}, the reconstruction of the stochastic
return map is presented for different values of $\lambda$ and a fixed noise amplitude
$\sigma = 0.01$. Notice that at the boundary between two sectors a trajectory of the
random differential equation is more likely to have two possible oscillations modes according to the noise
realization. This is a consequence of the fact that also the stochastic map is a
projection onto the $z$ axis and the noise affecting the other components may also
cause a jump to a different sector. The graph of the stochastic return map approximates
the graph of the set-valued map. From the reconstruction we find that there exists an
isolated minimal invariant set $E_\lambda$ for $\lambda_0^*< \lambda < \lambda_0 $ for some
$\lambda_0^*$. In $\lambda_0\approx -6.86$, we observe that the set  $E_{\lambda_0}$ satisfies the necessary conditions for having a bifurcation. Since it does not disappear for $\lambda > \lambda_0$ the set $E_{\lambda_0}$ blows up lower semi-continuously in the sense presented in \cite{LambRasmussenRodrigues}. The bifurcation can be observed in the reconstruction of the
minimal invariant set in Figure~\ref{figure:minimalinvariant}.
\begin{figure}[htbp]
\centering
\begin{overpic}[width=0.9\textwidth,unit=1mm]{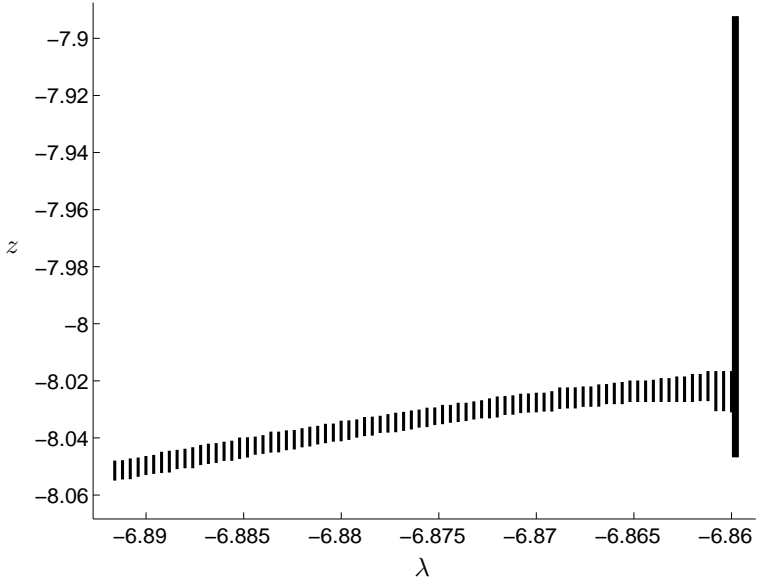}
\put(33,2){{\parbox{0.4\linewidth}{%
     \[
       \lambda
     \]}}}
\put(-21,45){{\parbox{0.4\linewidth}{%
     \[
        z
     \]}}}
\end{overpic}
\caption{\label{figure:minimalinvariant}Minimal invariant set of the stochastic return
map reconstructed for $\lambda \in [-6.9,-6.859]$. The bifurcation is observed in
$\lambda \approx -6.86$. }
\end{figure}
We evaluated the derivative of the extremal map $f^-$ associated to the stochastic return
map $p_\lambda^{st}$ with the algorithm presented in Section \ref{sec:algorithm}. The
methodology and the results are presented in Figure \ref{figure:koperderivative}. The
derivative of $f^-_\lambda$ evaluated at the boundary of $E_\lambda$ crosses $1$ from
below as $\lambda $ crosses $\lambda_0$. This can be used as early-warning signal due to
Proposition~\ref{prop:sufficient}.
\begin{figure}[ht!bp]
\begin{minipage}{0.5\textwidth}
\begin{overpic}[width=\textwidth,unit=1mm]{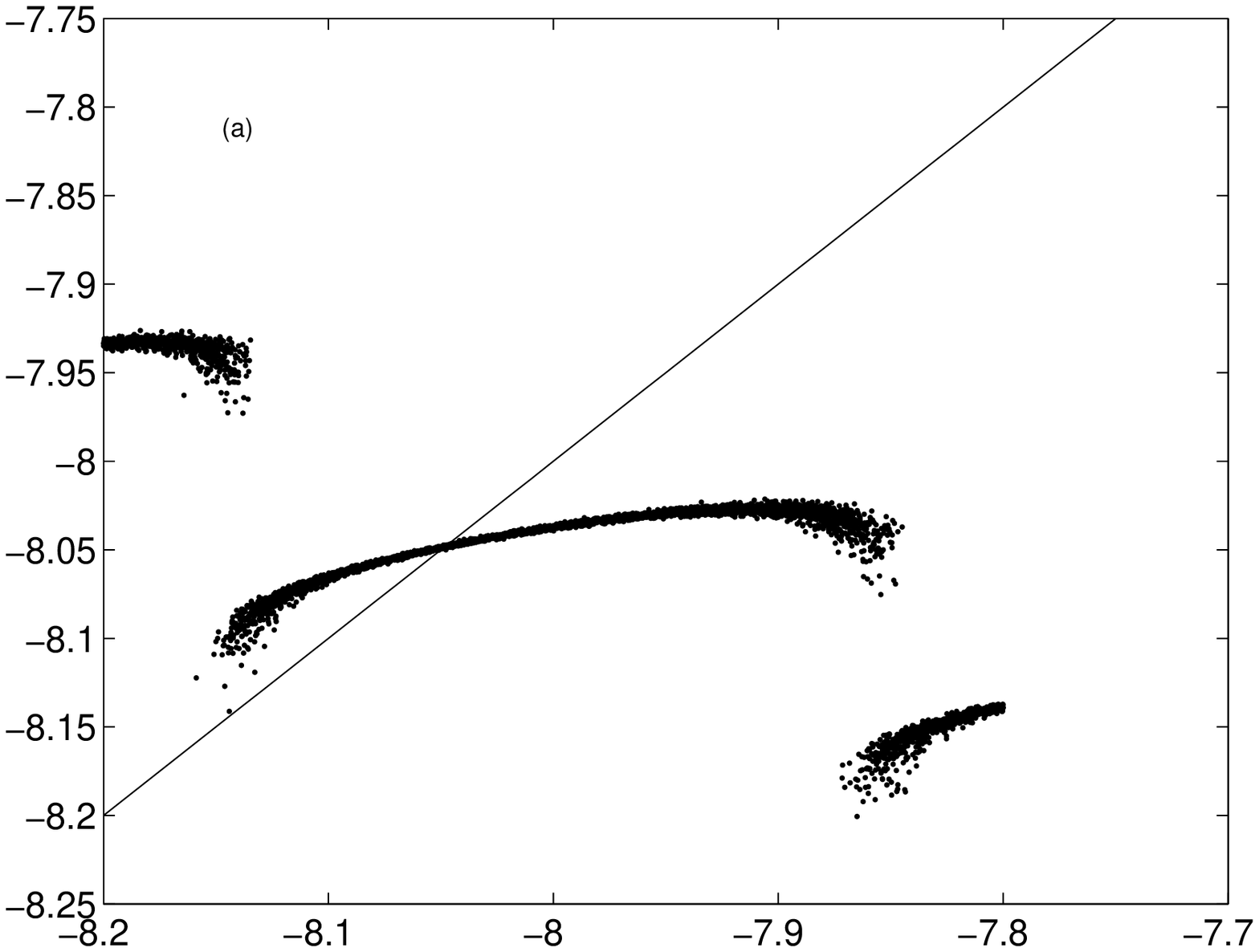}
\put(20,0){{\parbox{0.4\linewidth}{%
     \[
       z
     \]}}}
\put(-11,23){{\parbox{0.4\linewidth}{%
     \[
        p_\lambda
     \]}}}
\end{overpic}
\end{minipage}
\begin{minipage}{ 0.5\textwidth}
\begin{overpic}[width=\textwidth,unit=1mm]{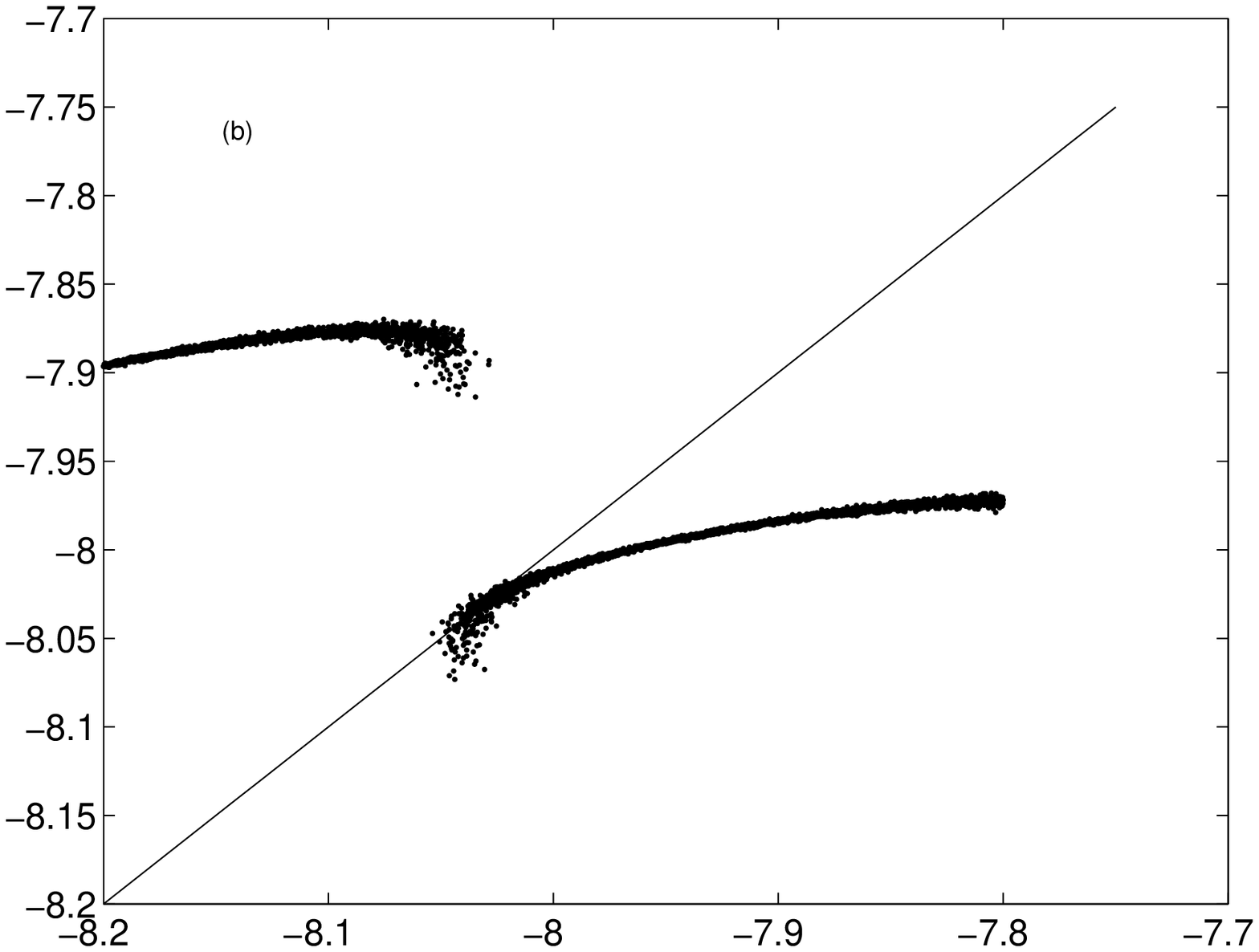}
\put(20,0){{\parbox{0.4\linewidth}{%
     \[
       z
     \]}}}
\put(-11,23){{\parbox{0.4\linewidth}{%
     \[
        p_\lambda
     \]}}}
\end{overpic}
\end{minipage}
\caption{\label{figure:koperlambdastochastic} (a) $\lambda = -6.9$.
Approximation of the set-valued dynamical system associated to the stochastic
return map before bifurcation. The minimal invariant set is easily identified by the
intersection of the identity line and the map. (b) $\lambda = -6.85$
Approximation of the set-valued dynamical system associated to the stochastic
return map after bifurcation. Notice that the system now can randomly jump
between two different modes of oscillations.}
\end{figure}

%%%%%%%%%%%%%%%%%%%%%%%%%%%%%%%%%%%%%%%%%%%%%%%%%%%%%%%%%%%%%%%%%%%%%%%%%%%%%%%%%%%%%%%%%%%%%%%%
%\bibliographystyle{amsalpha}
%\bibliography{./KMR_paper_v4}

\newcommand{\etalchar}[1]{$^{#1}$}
\providecommand{\bysame}{\leavevmode\hbox to3em{\hrulefill}\thinspace}
\providecommand{\MR}{\relax\ifhmode\unskip\space\fi MR }
\providecommand{\MRhref}[2]{%
  \href{http://www.ams.org/mathscinet-getitem?mr=#1}{#2}
}
\providecommand{\href}[2]{#2}

\end{document}